\documentclass[twoside,reqno]{amsart}
\usepackage{amssymb,amsmath,amsfonts,amsthm}
\usepackage{enumerate}
\usepackage{array}
\usepackage[colorlinks=true]{hyperref}
\hypersetup{urlcolor=blue, citecolor=red}

\newtheorem{theorem}{Theorem}

\newtheorem{lemma}[theorem]{Lemma}

\newcommand{\HH}{\mathbb H}
\newcommand{\R}{{\mathbb R}}

\newcommand{\p}{\partial}

\def\dt{\partial_t}

\def\Ac{\mathcal A }
\def\Bc{\mathcal B }
\def\Fc{\mathcal F }

\def\v{{\bf v}}
\def\B{{\bf B}}
\def\H{{\bf H}}
\def\E{{\bf E}}
\def\Lb{{\bf L}}
\def\Rb{{\bf R}}
\def\Sb{{\bf S}}

\def\ds{\displaystyle}
\def\eps{\varepsilon}
\def\vphi{\varphi}
\def\vpu{\varphi^{(1)}}
\def\vpd{\varphi^{(2)}}

\def\hvpu{\hat\varphi^{(1)}}
\def\hvpd{\hat\varphi^{(2)}}
\def\la{\lambda}
\def\a{\alpha}
\def\b{\beta}
\def\l{\ell}
\def\La{\Lambda}

\def\sla{\sigma(\lambda)}

\def\th{\theta}
\def\sgn{\rm sgn}
\def\intR{\int_{-\infty}^{+\infty}}
\def\intRneg{\int_{-\infty}^{0}}
\DeclareMathOperator{\dv}{{div}}%{{\nabla \cdot\, }}
\DeclareMathOperator{\rot}{{\nabla \times\, }}

\usepackage{mathtools}
\renewcommand*{\doteq }{\vcentcolon=} % := con punti centrati verticalmente

\title[Plasma-vacuum interface]{Nonlinear surface waves\\ on the plasma-vacuum interface}
\author[p. secchi]{Paolo Secchi}


\address{DICATAM, Mathematical Division, University of Brescia \\ Via Branze 43, 25123 Brescia, Italy}

\email{paolo.secchi@unibs.it}

\thanks{The author is supported by the national research project PRIN 2012 \lq\lq Nonlinear Hyperbolic Partial Differential Equations, Dispersive and Transport Equations: theoretical and applicative aspects\rq\rq.}

\date{\today}

\begin{document}

\begin{abstract}
In this paper we study the propagation of weakly nonlinear surface waves on a plasma-vacuum interface. In the plasma region we consider the equations of incompressible magnetohydrodynamics, while in vacuum the magnetic and electric fields are governed by the Maxwell equations. A surface wave propagate along the plasma-vacuum interface, when it is linearly weakly stable. 

Following the approach of \cite{ali-hunter}, we measure the amplitude of the surface wave by the normalized displacement of the interface in a reference frame moving with the linearized phase velocity of the wave, and obtain that it satisfies an asymptotic nonlocal, Hamiltonian evolution equation. We show the local-in-time existence of smooth solutions to the Cauchy problem for the amplitude equation in noncanonical variables, and we derive a blow up criterion.

\end{abstract}

\maketitle

\section{Introduction}
\label{sect1}

 {Plasma-vacuum interface problems appear in the mathematical modeling of plasma confinement by magnetic fields in thermonuclear energy production (as in Tokamaks; see, e.g., \cite{Goed}). In this model, the plasma is confined inside a perfectly conducting rigid wall and isolated from it by a region containing very low density plasma, which may qualify as vacuum, due to the effect of strong magnetic fields.
In Astrophysics, the plasma-vacuum interface problem can be used for modeling the motion of a star or the solar corona when magnetic fields are taken into account. 
% In this case, the vacuum region surrounding a plasma body is usually assumed to be unbounded.

This subject is very popular since the 1950--70's, but most of theoretical studies are devoted to finding stability criteria of equilibrium states. The typical work in this direction is the famous paper of Bernstein et al. \cite{BFKK}, where the plasma-vacuum interface problem is considered in its classical statement modeling the plasma confined inside a perfectly conducting rigid wall and isolated from it by a vacuum region. 

Assume that the plasma-vacuum interface is described by $\Gamma (t)=\{F(t,x)=0\}$, and that 
$\Omega^\pm(t)=\{F(t,x)\gtrless 0\}$ are the space-time domains occupied by the plasma and the vacuum respectively. Since $F$ is an unknown, this is a free-boundary problem.

In \cite{BFKK} (see also \cite{Goed}) the plasma is described by the  equations of ideal compressible Magneto-Hydro-dynamics (MHD)\footnote{In this introduction we don't write explicitly the compressible MHD equations that are not really needed, as in the sequel we are going to consider the incompressible MHD equations.}, whereas in the vacuum region one considers the so-called {\it pre-Maxwell dynamics}}
\begin{equation}
\nabla \times {H} =0,\qquad {\rm div}\, {H}=0,\label{6}
\end{equation}
\begin{equation}
\nabla \times {E} =- \frac1{c}\partial_t{H},\qquad {\rm div}\, E=0,\label{6'}
\end{equation}
describing the vacuum magnetic field ${H}\in\mathbb{R}^3$ and electric field ${E}\in\mathbb{R}^3$; $c$ is the speed of light. That is in the Maxwell equations one neglects the displacement current $(1/c)\,\partial_tE$.
From \eqref{6'} the electric field $E$ is a secondary variable that may be  computed from the magnetic field ${ H}$.

The dependent variables in the plasma region $\Omega^+(t)$ and in the vacuum region $\Omega^-(t)$ (i.e. the solution $H$ of \eqref{6}) are linked at the free interface by the boundary conditions
\begin{subequations}\label{7}
\begin{align}
\frac{{\rm d}F }{{\rm d} t}=0,\quad 
 [q]=0,\quad  B\cdot N=0,\label{8a}
 \\  H\cdot N=0 \label{8b}
\end{align}
 \end{subequations}
on $\Gamma (t)$, where $B\in\mathbb{R}^3$ denotes the magnetic field in the plasma region, $[q]$ denotes the jump of the total pressure across the interface, and $N=\nabla F$. The
first condition in \eqref{8a} (where $\frac{{\rm d}\, }{{\rm d} t}$ denotes the material derivative) means that the interface moves with the velocity of plasma particles at the boundary.

An important feature of the plasma-vacuum interface problem is that the uniform Kreiss-Lopatinskii condition \cite{kreiss} is never satisfied. The Kreiss-Lopatinskii condition may be violated, because there are modes that grow arbitrarily fast, and the interface is violently unstable as in the Kelvin-Helmholtz instability of a vortex sheet. Alternatively the Kreiss-Lopatinskii condition may be satisfied in weak form, and the interface is weakly but not strongly stable. In that case surface waves propagate along the discontinuity front.

Another important difficulty of the plasma-vacuum problem is that we cannot test the Kreiss-Lopatinski condition analytically, as for other free-boundary problems in MHD, so it is not known a complete description of the parameters set of violent instability / weak stability. Moreover, since the number of dimensionless parameters for the constant coefficients linearized problem is big, a complete numerical test of the Kreiss-Lopatinski condition seems unrealizable in practice. Thus it becomes important to investigate in a different way which stability conditions may ensure the weak stability of the problem.

Until recently, there were no well-posedness results for full ({\it non-stationary}) plasma-vacuum models. A basic a priori energy estimate for solutions of the linearized plasma-vacuum problem was first derived in \cite{trakhinin10}, under the stability condition stating that the magnetic fields, respectively $B$ and $H$, on either side of the interface are not collinear, i.e.
\begin{equation}
\begin{array}{ll}\label{stability}
B\times H\not=0 \qquad \mbox{on}\; \Gamma(t).
\end{array}
\end{equation}
The existence of solutions to the linearized problem was then proved in \cite{SeTr}. In \cite{mttpv} similar results are obtained for the plasma-vacuum problem in incompressible MHD.

In \cite{SeTr, trakhinin10}, for technical simplicity the moving interface $\Gamma (t)$ was assumed to have the form of a graph $F(t,x)=x_2-\varphi(t,x_1,x_3)$, i.e., both the plasma and vacuum domains are unbounded.
However, as was noted in the subsequent paper \cite{SeTrNl}, such form of the domains is not suitable for the original nonlinear free boundary problem because in that case the vacuum region $\Omega^-(t)$ is a simply connected domain. Indeed, in a simply connected domain the homogeneous elliptic problem \eqref{6}, \eqref{8b} has only the trivial solution $H=0$, and the whole problem is reduced to solving the MHD equations with a vanishing total pressure $q$ on $\Gamma (t)$.
The technically difficult case of multiply connected vacuum regions was postponed to a future work. 

Instead of this, in \cite{SeTrNl} the plasma-vacuum system is assumed to be not isolated from the outside world due to a given surface current on the fixed boundary of the vacuum region that forces oscillations. In laboratory plasmas this external excitation may be caused by a system of coils. This model can also be exploited for the analysis of waves in astrophysical plasmas, e.g., by mimicking the effects of excitation of MHD waves by an external plasma by means of a localized set of \lq\lq coils", when the response of the internal plasma is the main issue (see a more complete discussion in \cite{Goed}).

Under the above mentioned stability condition \eqref{stability}, in \cite{SeTrNl} the authors prove the local-in-time existence of a smooth solution in suitable anisotropic Sobolev spaces to the nonlinear plasma-vacuum interface problem, with the proof based on the results of \cite{SeTrNl} for the linearized problem, and a suitable Nash-Moser-type iteration. The stability condition $B\times H\not=0$ on $\Gamma(t)$ is assumed at time $t=0$ for the initial data and it is shown to persist for small positive time.

As in the classical formulation of the plasma-vacuum problem with the pre-Maxwell dynamics the displacement current is neglected and \eqref{6'} is considered {\it a posteriori} to recover the electric field from the magnetic field, the influence of the electric field is somehow hidden in the model. In order to investigate the influence of the vacuum electric field on the well-posedness of the problem, in \cite{CDAS,mandrik-trakhinin}, instead of the pre-Maxwell dynamics, in the vacuum region the authors don't neglect the displacement current and consider the complete system of {\it Maxwell equations} for the electric and the magnetic fields.

Indeed, for the relativistic plasma-vacuum problem, Trakhinin \cite{trakhinin12} has shown the possible ill-posedness in the presence of a sufficiently strong vacuum electric field. Since relativistic effects play a rather passive role in the analysis of \cite{trakhinin12}, it is natural to expect the same for the nonrelativistic problem.
In \cite{CDAS,mandrik-trakhinin} the authors show that a {\it sufficiently weak} vacuum electric field, under the same stability condition \eqref{stability}, precludes ill-posedness and gives the well-posedness of the linearized problem.

In this paper we are interested to investigate the well-posedness of the problem when \eqref{stability} is violated, i.e. when the magnetic fields on either side of the interface are collinear. For the sake of simplicity we consider the plasma-vacuum interface problem in two-dimensions, with the coupling of the incompressible MHD equations in the plasma region and the Maxwell equations in the vacuum region. The solution is close to a stationary basic state with parallel magnetic fields at the flat interface.

%%%%%%%%%%%%%%%%%%%%%%%%%%%%%%%%%%%%%%%%%%%%%%%%%%%%%%
%%%%%%%%%%%%%%%%%%%%%%%%%%%%%%%%%%%%%%%%%%%%%%%%%%%%%%
%%%%%%%%%%%%%%%%%%%%%%%%%%%%%%%%%%%%%%%%%%%%%%%%%%%%%%
%%%%%%%%%%%%%%%%%%%%%%%%%%%%%%%%%%%%%%%%%%%%%%%%%%%%%%

%\bigskip

To study the time evolution of the plasma-vacuum interface we follow the approach of  \cite{ali-hunter} and
we show that, in a unidirectional surface wave, the normalized displacement $x_2=\vphi(t,x_1)$ of a weakly stable surface wave along the interface, in a reference frame moving with the linearized phase velocity of the wave, satisfies the quadratically nonlinear, nonlocal asymptotic equation 
\begin{equation}
\begin{array}{ll}\label{equ1}
\vphi_t+\frac12 \HH[\Phi^2]_{xx} + \Phi\vphi_{xx}=0, \qquad \Phi=\HH[\vphi] \, .
\end{array}
\end{equation}
Here $\HH$ denotes the Hilbert transform defined by
\begin{equation*}
\begin{array}{ll}\label{}
\ds \HH[\vphi](x)=\frac1{\pi}\,\mbox{p.v.}\intR\frac{\vphi(y)}{x-y}\, dy, 
\end{array}
\end{equation*}
and such that
\[
\HH[e^{ikx}]=
-i \, \sgn(k)\, e^{ikx}\, , \qquad \Fc[\HH[\vphi]]=i \, \sgn(k)\, \Fc[\vphi],
\]
for $\Fc$ denoting the Fourier transformation.
%
%
%The quadratic nonlocal operator $\mathcal{Q}$ in \eqref{equ1} is defined through the spectral form of the equation that reads 
%\begin{equation}
%\begin{array}{ll}\label{equ2}
%\ds \hvphi_t(k,t)+i \ {\sgn}(k)\int_{-\infty}^{\infty}\Lambda(k-\ell,\ell)\hvphi(k-\ell,t)\hvphi(\ell,t)d\ell=0\, ,
%\end{array}
%\end{equation}
%where $\hvphi(k,t)=\Fc[\vphi](k,t)$ is the Fourier transform of the displacement $\vphi(x,t)$, and the kernel $\Lambda(k,\ell)$ is defined by
%\begin{equation}
%\begin{array}{ll}\label{lambda2}
%\ds \La(k,\ell)= \frac{2|k+\l|\, |k|\, |\l|}{|k+\l|+|k|+|\l|}\,.
%\end{array}
%\end{equation}
Equation \eqref{equ1} coincides with the amplitude equation for nonlinear Rayleigh waves \cite{hamilton-et-al} and current-vortex sheets in incompressible MHD \cite{ali-hunter,ali-hunter-parker}. It is interesting that exactly the same equation appears for the incompressible plasma-vacuum interface problem, where in the vacuum part the electric and magnetic fields are ruled by the Maxwell equations. Equation \eqref{equ1} also admits the other following spatial form
\begin{equation}
\begin{array}{ll}\label{equ1bis}
\vphi_t+[ \HH,\Phi] \Phi_{xx} + \HH[\Phi^2_x]=0 \, ,
\end{array}
\end{equation}
where $[ \HH,\Phi]  $ is the commutator of $\HH$ with multiplication by $\Phi$, see \cite{hunter11}. This form of \eqref{equ1bis} shows that there is a cancelation of the second order spatial derivatives appearing in \eqref{equ1}.

By adapting the proof of \cite{hunter2} we show the local-in-time existence of smooth solutions to the Cauchy problem for amplitude equation in noncanonical variables, and we derive a blow up criterion. Numerical computations \cite{ali-hunter, hamilton-et-al} show that solutions of \eqref{equ1} form singularities in which the derivative $\vphi_x$ blows up, but $\vphi$ appears to remain continuous. As far as we know, the global existence of appropriate weak solutions is an open question.

The paper is organized as follows. In Sec. \ref{pbm} we formulate the plasma-vacuum problem for incompressible MHD equations in the plasma region, Maxwell equations in the vacuum region and suitable jump conditions on the free interface. In Sec. \ref{asymptotic} we introduce the asymptotic expansion for small-amplitude, long-time weakly nonlinear surface waves.
In Sec. \ref{first order} we solve the equations for the first order term of the asymptotic expansion. This first order solution depends on an arbitrary wave profile function. In Sec. \ref{second order} we solve the second order perturbation equations. When the second order solution of the interior equations is substituted in the second order jump conditions, one gets a linear system whose resolution is obtained under solvability conditions leading to the amplitude equation \eqref{equ1}. The arbitrary wave profile function of Sec. \ref{first order} is then determined as the solution of this amplitude equation.
The results of Sections \ref{asymptotic} to \ref{second order} are summarized in Theorem \ref{summa}.
In Sec. \ref{noncanonical} we prove the local in time existence of a smooth solution of an initial value problem for a noncanonical form of \eqref{equ1}, see Theorem \ref{wellp}, and derive a blow up criterion, see Lemma \ref{blow}.
 
%\vfill
%\eject

\section{The plasma-vacuum interface problem}\label{pbm}

We consider the equations of incompressible magneto-hydrodynamics (MHD), i.e. the equations governing the motion 
of a perfectly conducting inviscid incompressible plasma. In the case of a homogeneous plasma (the density $\rho 
\equiv$ const $>0$), the equations in a dimensionless form read:
\begin{equation}
\label{mhd1}
\begin{cases}
\partial_t \v +\nabla \cdot 
(\v \otimes \v-\B\otimes \B) +\nabla q =0 \, ,\\
\partial_t \B -\nabla \times
(\v\times \B) =0 \, ,\\
\dv \v=0\,, \; \dv \B=0 \, ,
\end{cases}
\end{equation}
where $\v $ denotes the plasma velocity, $\B$ is the magnetic field (in Alfv\'en velocity 
units), $q=p+|\B|^2/2$ is the total pressure, $p$ being the pressure.

For smooth solutions, system \eqref{mhd1} can be written in equivalent form as a symmetric system
\begin{equation}
\label{mhd2}
\begin{cases}
\dt \v +(\v\cdot\nabla)\v -(\B\cdot\nabla)\B  +\nabla q=0\, ,\\
\dt \B +(\v\cdot\nabla)\B -(\B\cdot\nabla)\v  =0\, ,\\
\dv \v=0 \, .
\end{cases}
\end{equation}
In addition the magnetic field must satisfy the constraint 
\[
 \dv \B=0\,,
\]
which is preserved by the evolution in time if it is satisfied by the initial data.

Let $\Omega^+(t)$ and $\Omega ^-(t)$ be space-time domains occupied by the plasma and the vacuum respectively, separated by an interface $\Gamma (t)$. That is, in the domain
$\Omega^+(t)$ we consider system \eqref{mhd2} governing the motion of the plasma and in the domain $\Omega^-(t)$  we have the Maxwell system
\begin{equation}\label{eq:Maxwell}
\begin{cases}
\nu\p_t \H + \rot \E = 0  \,, \\
\nu\p_t \E - \rot \H = 0  \,, 
\end{cases}
\end{equation}
describing the vacuum magnetic and electric fields $\H,\E\in\mathbb{R}^3$. 
Here, the equations are written in nondimensional form through a suitable scaling (see Mandrik--Trakhinin~\cite{mandrik-trakhinin}), and $\nu=\frac{\bar{v}}{c}$, where $\bar{v}$ is the velocity of a uniform flow and $c$ is the speed of light in vacuum. If we choose $\bar{v}$ to be the speed of sound in vacuum, we have that $\nu$ is a small, even though fixed parameter.
System \eqref{eq:Maxwell} is supplemented by the divergence constraints
\begin{equation*}
 \dv \H = \dv \E = 0
\label{}
\end{equation*}
on the initial data.
The plasma variables are connected with the vacuum magnetic and electric fields on the interface $\Gamma (t)$
through the relations \cite{BFKK,Goed}
\begin{equation}
\sigma = \v\cdot N , \quad[q]=0,\quad  \B\cdot N= \H\cdot N=0 ,\quad  N \times \E = \nu(\v \cdot N)  \H\quad \mbox{on}\ \Gamma (t),
\label{8}
\end{equation}
where $\sigma$ denotes the velocity of 
propagation of the interface $\Gamma (t)$,
$N$ is a normal vector and $[q]= q|_{\Gamma}-\frac{1}{2}|{\H}|^2_{|\Gamma}+\frac12|\E|^2_{|\Gamma}$ is the jump of the total pressure across the interface.

%%%%%%%%%%%%%%%%%%%%%%%%%%%%%%%%%%%
We consider the case of two space dimensions and write
\[
\v=(v_1,v_2)^T, \qquad \B=(B_1,B_2)^T.
\]
In the (three-dimensional) Maxwell equations \eqref{eq:Maxwell} we assume that 
\[
\H=(H_1,H_2,0)^T, 
\]
and that there is no dependence of $\H$ on the third space variable $x_3$. It follows from \eqref{eq:Maxwell} that $\E$ takes the form
\[
 \E=(0,0,E)^T,
\]
and the Maxwell equations
reduce to 
\begin{equation}\label{eq:Maxwell2}
\begin{cases}
\nu\p_t H_1 + \p_2E = 0  \,, \\
\nu\p_t H_2 - \p_1E = 0  \,, \\
\nu\p_t E - \p_1 H_2 +\p_2 H_1 = 0  \,, 
\end{cases}
\end{equation}
under the constraint
\begin{equation*}
\p_1H_1+\p_2H_2 =  0
\label{}
\end{equation*}
on the initial data. From now on we write 
\[
\H=(H_1,H_2)^T, 
\]
hoping that this small abuse of notation will create no confusion for the reader.

%%%%%%%%%%%%%%%%%%%%%%%%%%%%%%%%%%%%%%%%%%

Let us assume that the moving interface $\Gamma(t)$ takes the form $$
\Gamma(t) \doteq  \{ (x_1,x_2) \in \R^2 \, , \, x_2=\zeta(x_1,t)\} \, ,
$$
where $t \in [0,T]$. Then we have $\Omega^\pm(t)=\{x_2\gtrless \zeta(x_1,t)\}$.
With our parametrization of $\Gamma (t)$,  the boundary conditions \eqref{8} at the interface reduce to
\begin{equation}
\begin{array}{ll}
\label{bc}
\partial_t\zeta =\v\cdot N\, ,\quad q=\frac{1}{2}\left( H_1^2+H_2^2-E^2\right)\, ,\\
\\
\B\cdot N=0\, , \quad \H\cdot N=0 \, ,
\quad
 E-\nu\zeta_tH_1=0 \quad \mbox{on}\quad \Gamma (t)\, ,

\end{array}
\end{equation}
where $N=(-\partial_1\zeta ,1 )$.

%%%%%%%%%%%%%%%%%%%%%%%%%%%%%%%%%%%%%%%%%%%%%%%%
A stationary solution of \eqref{mhd2}, \eqref{eq:Maxwell2}, \eqref{bc} with interface located at $\{x_2=0\}$ is given by the constant states
\[
\v^0=(v_1^0,0)^T, \qquad \B^0=(B_1^0,0)^T,
\]
\[
\H^0=(H_1^0,0)^T, \qquad E^0=0, \qquad q^0=\frac{1}{2} (H_1^0)^2.
\]
We will consider the propagation of surface waves that are localized near the interface. The corresponding solutions must satisfy the decay conditions
\begin{equation}
\begin{array}{ll}\label{decay1}
\ds \lim_{x_2\to+\infty}(\v,\B,q)=U^0\doteq(v_1^0,0,B_1^0,0,q^0)\, , \\
\\
\ds \lim_{x_2\to-\infty}(\H,E)=V^0\doteq(H_1^0,0,0)\, .
\end{array}
\end{equation}

%%%%%%%%%%%%%%%%%%%%%%%%%%%%%%%%%%%%%%%%%%%%%%%%%%%%%%%
\section{The asymptotic expansion}\label{asymptotic}

As in \cite{ali-hunter} we suppose that the perturbed interface has a slope of the order $\eps$, where $\eps$ is a small parameter. With respect to dimensionless variables in which the wavelength of the perturbation and the velocity of the surface wave are of the order one, the time scale for quadratically nonlinear effects to significantly alter the wave profile is of the order $\eps^{-1}$. We therefore introduce a \lq\lq slow\rq\rq time variable $\tau=\eps t$. We also introduce a spatial variable $\th=x_1-\la t$ in a reference frame moving with the surface wave. Here, $\la$ is the linearized phase velocity of the wave, which we will determine as part of the solution.

We write the perturbed location of the interface as 
\[
x_2=\eps\vphi(\th,\tau;\eps),
\]
and define a new independent variable
\[
\eta=x_2-\eps\vphi(\th,\tau;\eps),
\]
so that the perturbed interface is located at $\eta=0$. We look for an asymptotic expansion of the solution $U=(\v,\B,q)^T,\ V=(\H,E)^T$ and $\vphi$ as $\eps\to0$ of the form 
\begin{equation}
\begin{array}{ll}\label{expansion}
U(\th,\eta,\tau;\eps)=U^0+\eps U^{(1)}(\th,\eta,\tau)+\eps^2 U^{(2)}(\th,\eta,\tau) + O(\eps^3), \qquad  \eta>0,
\\
V(\th,\eta,\tau;\eps)=V^0+\eps V^{(1)}(\th,\eta,\tau)+\eps^2 V^{(2)}(\th,\eta,\tau) + O(\eps^3), \qquad  \eta<0,
\\
\vphi(\th,\tau;\eps)= \vphi^{(1)}(\th,\tau)+\eps \vphi^{(2)}(\th,\tau) + O(\eps^2).

\end{array}
\end{equation}
We expand the partial derivatives with respect to the original time and space variables as
\begin{equation*}
\begin{array}{ll}\label{}
\p_t=-\la\p_\th+\eps(\p_\tau+\la\vphi_\th\p_\eta)-\eps^2\vphi_\tau\p_\eta,\\
\p_{x_1}=\p_\th-\eps\vphi_\th\p_\eta,\\
\p_{x_2}=\p_\eta.
\end{array}
\end{equation*}
We substitute these expansions in \eqref{mhd2}, \eqref{eq:Maxwell2}, Taylor expand the result with respect to $\eps$ and equate coefficients of $\eps^1$ and $\eps^2$ to zero. In the interior the asymptotic solution satisfies at the first order
%%%%%%%%%%%%%%%%%%%%%%%%%%%%%%%
\begin{equation}
\begin{cases}\label{equ1order}
(\la-v^0_1)\p_\th v^{(1)} + B^0_1\p_\th B^{(1)}-\begin{pmatrix}
 \p_\th \\
 \p_\eta
\end{pmatrix}q^{(1)}=0\, ,\\
(\la-v^0_1)\p_\th B^{(1)} + B^0_1\p_\th v^{(1)}=0\, ,\\
\p_\th v_1^{(1)}+\p_\eta v_2^{(1)}=0\, ,&\qquad\mbox{for }  \eta>0,\\
\end{cases}
\end{equation}
%%%%%%%%%%%%
\begin{equation}
\begin{cases}\label{equ12order}
\nu\la\p_\th H_1^{(1)}-\p_\eta E^{(1)}=0\, ,\\
\nu\la\p_\th H_2^{(1)}+\p_\th E^{(1)}=0\, ,\\
\nu\la\p_\th E^{(1)}+\p_\th H_2^{(1)}-\p_\eta H_1^{(1)}=0\, ,&\qquad\mbox{for }  \eta<0.
\end{cases}
\end{equation}
We expand the jump conditions in \eqref{bc}, with $\zeta=\eps\vphi$, and equate coefficients of $\eps^1$ and $\eps^2$ to zero. We find that the solutions satisfy at the first order the following jump conditions
%%%%%%%%
\begin{equation}
\begin{cases}\label{bc1order}
\ds (\la-v^0_1)\p_\th \vpu + v_2^{(1)}=0\, ,\\
\ds B^0_1\p_\th \vpu - B_2^{(1)}=0\, , \qquad
H^0_1\p_\th \vpu - H_2^{(1)}=0\, ,\\
\ds q^{(1)}=H_1^0H_1^{(1)}\, , \qquad
E^{(1)}+\nu\la H_1^0\p_\th \vpu=0 \, ,&\qquad\mbox{for }  \eta=0.
\end{cases}
\end{equation}
%%%%%%%%%%%%%%%%%%%
At the second order we obtain
\begin{equation}
\begin{cases}\label{equ2order}
(\la-v^0_1)\p_\th v^{(2)} + B^0_1\p_\th B^{(2)}-\begin{pmatrix}
 \p_\th \\
 \p_\eta
\end{pmatrix}q^{(2)}=p_1\, ,\\
(\la-v^0_1)\p_\th B^{(2)} + B^0_1\p_\th v^{(2)}=p_2\, ,\\
-\p_\th v_1^{(2)}-\p_\eta v_2^{(2)}=p_3\, ,&\qquad\mbox{for }  \eta>0,\\
\end{cases}
\end{equation}
\begin{equation}
\begin{cases}\label{equ22order}
\nu\la\p_\th H_1^{(2)}-\p_\eta E^{(2)}=p_1'\, ,\\
\nu\la\p_\th H_2^{(2)}+\p_\th E^{(2)}=p_2'\, ,\\
\nu\la\p_\th E^{(2)}+\p_\th H_2^{(2)}-\p_\eta H_1^{(2)}=p_3'\, ,&\qquad\mbox{for }  \eta<0,
\end{cases}
\end{equation}
and the jump conditions
\begin{equation}
\begin{cases}\label{bc2order}
\ds (\la-v^0_1)\p_\th \vpd + v_2^{(2)}=r_1\, ,\\
\ds B^0_1\p_\th \vpd - B_2^{(2)}=r_2\, , \qquad
H^0_1\p_\th \vpd - H_2^{(2)}=r_3\, ,\\
\ds q^{(2)}-H_1^0H_1^{(2)}=r_4\, , \qquad
E^{(2)}+\nu\la H_1^0\p_\th \vpd=r_5 \, ,&\qquad\mbox{for }  \eta=0,
\end{cases}
\end{equation}
where we have denoted
\begin{equation*}
\begin{array}{ll}\label{}
p_1\doteq (\p_\tau+\la\vpu_\th\p_\eta)v^{(1)} + (v_1^{(1)}\p_\th+v_2^{(1)}\p_\eta-v_1^0\vpu_\th\p_\eta)v^{(1)}\\
\quad\qquad - (B_1^{(1)}\p_\th+B_2^{(1)}\p_\eta-B_1^0\vpu_\th\p_\eta)B^{(1)} - \begin{pmatrix}
\vpu_\th\p_\eta q^{(1)} \\
0 
\end{pmatrix}\, ,
\\
p_2\doteq (\p_\tau+\la\vpu_\th\p_\eta)B^{(1)} + (v_1^{(1)}\p_\th+v_2^{(1)}\p_\eta-v_1^0\vpu_\th\p_\eta)B^{(1)}\\
\quad\qquad - (B_1^{(1)}\p_\th+B_2^{(1)}\p_\eta-B_1^0\vpu_\th\p_\eta)v^{(1)} \, ,
\\
p_3\doteq -\vpu_\th\p_\eta v^{(1)}_1 \, ,
\end{array}
\end{equation*}
\begin{equation*}
\begin{array}{ll}\label{}
p_1'\doteq \nu(\p_\tau+\la\vpu_\th\p_\eta)H_1^{(1)} \, ,
\qquad
p_2'\doteq \nu(\p_\tau+\la\vpu_\th\p_\eta)H_2^{(1)} + \vpu_\th\p_\eta E^{(1)} \, ,
\\
p_3'\doteq \nu(\p_\tau+\la\vpu_\th\p_\eta)E^{(1)} + \vpu_\th\p_\eta H^{(1)}_2 \, ,
\end{array}
\end{equation*}
\begin{equation*}
\begin{array}{ll}\label{}
r_1\doteq (\p_\tau+v_1^{(1)}\p_\th)\vpu \, ,
\qquad
&r_2\doteq -B_1^{(1)}\p_\th \vpu \, ,
\\
r_3\doteq -H_1^{(1)}\p_\th \vpu \, ,
\qquad
&\ds r_4\doteq \frac1{2}\left( |H^{(1)}|^2 - (E^{(1)})^2\right) \, ,
\\
r_5\doteq -\nu\la H_1^{(1)}\p_\th \vpu + \nu H_1^0\p_\tau\vpu\, .
\end{array}
\end{equation*}
In the rest of the paper we solve equations \eqref{equ1order}--\eqref{bc2order}.
%%%%%%%%%%%%%%%%%%%%%%%%%%%%%%%%%%%%%%%%%%%%%%%%%%%%%%%

%%%%%%%%%%%%%%%%%%%%%%%%%%%%%%%%%%%%%%%%%%%%%%%%%%%%%%%
\section{The first order equations}\label{first order}

Introducing the Fourier transforms
\[
\hat{U}^{(1)}(k,\eta,\tau)=\frac1{2\pi}\intR U^{(1)}(\th,\eta,\tau)e^{-ik\th}d\th,
\]
\[
\hat{V}^{(1)}(k,\eta,\tau)=\frac1{2\pi}\intR V^{(1)}(\th,\eta,\tau)e^{-ik\th}d\th,
\]
\[
\hat{\vphi}^{(1)}(k,\tau)=\frac1{2\pi}\intR \vphi^{(1)}(\th,\tau)e^{-ik\th}d\th,
\]
and Fourier transforming \eqref{equ1order}--\eqref{bc1order} with respect to $\th$, we find the equations
%%%%%%%%%%%%%%%%%%%%%%%%%%%%%%%
\begin{equation}
\begin{cases}\label{equ1orderFourier}
(\la-v^0_1)ik\hat v^{(1)} + ik B^0_1\hat B^{(1)}-\begin{pmatrix}
ik \\
 \p_\eta
\end{pmatrix}\hat q^{(1)}=0\, ,\\
(\la-v^0_1)ik \hat B^{(1)} + ik B^0_1\hat v^{(1)}=0\, ,\\
ik\hat v_1^{(1)}+\p_\eta\hat v_2^{(1)}=0\, ,&\qquad\mbox{for }  \eta>0,\\
\end{cases}
\end{equation}
%%%%%%%%%%%%
\begin{equation}
\begin{cases}\label{equ12orderFourier}
\nu\la ik\hat H_1^{(1)}-\p_\eta\hat E^{(1)}=0\, ,\\
\nu\la ik\hat H_2^{(1)}+ik\hat E^{(1)}=0\, ,\\
\nu\la ik\hat E^{(1)}+ik\hat H_2^{(1)}-\p_\eta\hat H_1^{(1)}=0\, ,&\qquad\mbox{for }  \eta<0,
\end{cases}
\end{equation}
%%%%%%%%
\begin{equation}
\begin{cases}\label{bc1orderFourier}
\ds (\la-v^0_1)ik \hvpu + \hat v_2^{(1)}=0\, ,\\
\ds ikB^0_1 \hvpu - \hat B_2^{(1)}=0\, , \qquad
ikH^0_1 \hvpu - \hat H_2^{(1)}=0\, ,\\
\ds \hat q^{(1)}=H_1^0\hat H_1^{(1)}\, , \qquad
\hat E^{(1)}+\nu\la ikH_1^0 \hvpu=0 \, ,&\qquad\mbox{for }  \eta=0.
\end{cases}
\end{equation}
%%%%%%%%%%%%%%%%%%%
Let us first consider problem \eqref{equ1orderFourier}, that we write in the form\footnote{The choice of the symmetric form of equations \eqref{mhd2}, rather than the conservative form \eqref{mhd1} as in \cite{ali-hunter}, reflects in a different definition of the matrices $\Ac,\Bc$, and partly simplifies the following resolution.}
\begin{equation}
\begin{array}{ll}\label{equ1orderFourier2}
ik\Ac \hat{U}^{(1)}+\Bc\p_\eta \hat{U}^{(1)}=0,
\end{array}
\end{equation}
where the real symmetric matrices $\Ac,\Bc$ are defined by
\[
\Ac=\begin{pmatrix}
\la-v^0_1 & 0 &B^0_1 & 0 & -1\\
0 & \la-v^0_1& 0 & B^0_1 & 0 \\
B^0_1 & 0 & \la-v^0_1  & 0 & 0\\
0 & B^0_1 & 0 & \la-v^0_1  & 0 \\
-1& 0 & 0 & 0 & 0
\end{pmatrix} \, , \qquad
\Bc=\begin{pmatrix}
0 & 0 &0 & 0 & 0\\
0 & 0& 0 & 0 & -1 \\
0 & 0 & 0  & 0 & 0\\
0 & 0 & 0 & 0  & 0 \\
0& -1 & 0 & 0 & 0
\end{pmatrix} \, .
\]
As in \cite{ali-hunter} we compute an eigenvector $\Rb$ from 
$
(i\Ac -\Bc)\Rb=0.
$ After a convenient choice of normalization, this eigenvector is given explicitly by
\begin{equation}
\begin{array}{ll}\label{defR}
\Rb=(\la-v^0_1, i( \la-v^0_1) , -B^0_1 , -iB^0_1 ,  d)^T, \qquad\mbox{where   }\;   d\doteq (\la-v^0_1)^2- (B^0_1)^2.
\end{array}
\end{equation}
The general solution of \eqref{equ1orderFourier2} is
\begin{equation*}
\hat{U}^{(1)}(k,\eta,\tau)=a(k,\tau)e^{-k\eta}\Rb + b(k,\tau)e^{k\eta}\overline{\Rb}\, ,
\end{equation*}
where $a(k,\tau)$ and $b(k,\tau)$ are arbitrary complex-valued functions, the bar denotes a complex conjugate. The condition \eqref{decay1} at infinity implies
\begin{equation}\label{decay2}
\ds \lim_{\eta\to+\infty}\hat{U}^{(1)}(k,\eta,\tau)=0\, ; \\
\end{equation}
then we find
\begin{equation}
\begin{array}{ll}\label{sol1}
\hat{U}^{(1)}(k,\eta,\tau)=\begin{cases}
a(k,\tau)e^{-k\eta}\Rb, &\qquad\mbox{if }  k>0\, ,\\
b(k,\tau)e^{k\eta}\overline{\Rb},&\qquad\mbox{if }  k<0\, .
\end{cases}
\end{array}
\end{equation}
%%%%%%%%%%%%%%%%%%%%%%%%%%%%%%%%%%%%%%%%%%%%%%%%%%%%%%%

Let us consider now problem \eqref{equ12orderFourier} for $\eta<0$. Here we must work differently than before. From the second equation in \eqref{equ12orderFourier}, 
$
 \hat H_2^{(1)}=-\hat E^{(1)}/\nu\la$, and substituting in the other equations of \eqref{equ12orderFourier} we get
\begin{equation}
\begin{array}{ll}\label{equE}
 \p_\eta^2 \hat E^{(1)}+k^2(\nu^2\la^2-1)\hat E^{(1)}=0.
\end{array}
\end{equation}
In order to have 
\begin{equation}\label{decay3}
\ds \lim_{\eta\to-\infty}\hat{V}^{(1)}(k,\eta,\tau)=0\,  \\
\end{equation}
(obtained from \eqref{decay1}), we need to prescribe in \eqref{equE}
\begin{equation}
\begin{array}{ll}\label{condlambda}
\nu|\la|<1.
\end{array}
\end{equation}
The general solution of \eqref{equE} is
\begin{equation}\label{genersol}
\hat{E}^{(1)}(k,\eta,\tau)=\a(k,\tau)e^{\sla k\eta} + \b(k,\tau)e^{-\sla k\eta}\, ,
\end{equation}
where $\a(k,\tau)$ and $\b(k,\tau)$ are arbitrary complex-valued functions and 
\[
\sigma(\la)\doteq\sqrt{1-\nu^2\la^2}.
\]
From \eqref{equ12orderFourier}, \eqref{genersol} the general solution for the other unknowns is
\begin{equation}\label{genersol2}
\begin{array}{ll}
\hat{H}_1^{(1)}(k,\eta,\tau)=\frac{\sla}{i\nu\la}\left\{\a(k,\tau)e^{\sla k\eta} - \b(k,\tau)e^{-\sla k\eta}\right\}\, ,\\
\hat{H}_2^{(1)}(k,\eta,\tau)=-\frac{1}{\nu\la}\left\{\a(k,\tau)e^{\sla k\eta} + \b(k,\tau)e^{-\sla k\eta}\right\}\, .
\end{array}
\end{equation}
Finally, imposing the condition \eqref{decay3} at infinity to \eqref{genersol}, \eqref{genersol2} we find that
\begin{equation}
\begin{array}{ll}\label{sol2}
\hat{V}^{(1)}(k,\eta,\tau)=\begin{cases}
\a(k,\tau)e^{\sla k\eta}\begin{pmatrix}
 -i {\sla}/{\nu\la} \\
-{1}/{\nu\la}\\
1 
\end{pmatrix}, &\qquad\mbox{if }  k>0\, ,\\
\\
\b(k,\tau)e^{-\sla k\eta}\begin{pmatrix}
 i{\sla}/{\nu\la} \\
- {1}/{\nu\la}\\
1 
\end{pmatrix},&\qquad\mbox{if }  k<0\, .
\end{cases}
\end{array}
\end{equation}

%%%%%%%%%%%%%%%%%%%%%%%%%%%%%%%%%%%%%%%%%%%%%%%%%%%%%%%

Next, we use the solution \eqref{sol1}, \eqref{sol2} in the jump conditions \eqref{bc1orderFourier}.
First we consider the case $k>0$.
Under the assumption $\la-v^0_1\not=0$ or  $B^0_1\not=0$, the resulting equations may be written as a linear system 
for the unknowns $(a,\a,k\hvpu)$:
\begin{equation}
\begin{array}{ll}\label{systemjump}
\begin{pmatrix}
1&0 &1 \\
 0&1 &i\nu\la  H^0_1\\
 d& i{\sla}H^0_1/\nu\la&0
\end{pmatrix}
\begin{pmatrix}
a \\
\a \\
k\hvpu
\end{pmatrix}=0.
\end{array}
\end{equation}
This system has a nontrivial solution if 
\begin{equation}
\begin{array}{ll}\label{condlamda2}
d = (\la-v^0_1)^2- (B^0_1)^2=(H^0_1)^2\sla.
\end{array}
\end{equation}
%%%%

We discuss the possible real roots $\la$ of \eqref{condlamda2} that also satisfy \eqref{condlambda}.
%\begin{lemma}\label{lemma0}
%If $v^0_1,B^0_1$ are sufficiently small, and $|H^0_1|$ is sufficiently large, there exist two distinct real roots $-1<\la_1<0<\la_2<1$ of \eqref{condlamda2}.
%\end{lemma}
%%%%
\begin{lemma}\label{lemma1}
\begin{enumerate}
\item If $|B^0_1|>|v^0_1|+1/\nu$, equation \eqref{condlamda2} does not have any real root.
\item If $ |B^0_1|=|v^0_1|+1/\nu$, for all $|H^0_1|>0$ and $v^0_1\not=0$ there exists one real root $\la=-\sgn(v^0_1)/\nu$. If $v^0_1=0$ then $\la=\pm1/\nu$. Thus in any case $|\la|=1/\nu$.
\item If $|v^0_1|-1/\nu \leq |B^0_1|<|v^0_1|+1/\nu$, for all $|H^0_1|>0$ there exist one or two real roots $\la$ of \eqref{condlamda2} such that $|\la|<1/\nu$.
%(and possibly another real root with $|\la|=1$).
\item If $ |B^0_1|<|v^0_1|-1/\nu$, there exists $H^*>0$ such that, for all $|H^0_1|\ge H^*$, there exist two real roots $\la$ of \eqref{condlamda2} such that $|\la|<1/\nu$ (coincident roots if $|H^0_1|= H^*$); if $|H^0_1|<H^*$ \eqref{condlamda2} does not have any real root.
\end{enumerate}
\end{lemma}
%\fbox{potrebbe bastare uno solo di questi valori tra 1 e -1}
%The proof of Lemma \ref{lemma1} is given in Appendix \ref{App:A}. 
Observe that for all such $|\la|<1/\nu$, from \eqref{condlamda2} there holds $\la\not=v^0_1$ and $\la\not=v^0_1\pm B^0_1$, i.e. $d\not=0$.
\begin{proof}%[Proof of Lemma~\ref{lemma1}]
%Without loss of generality we may assume $B^0_1>0,H^0_1>0$. 
The roots of \eqref{condlamda2} are given by the points of intersection in the plane $\la,y$ of the parabola $y=(\la-v^0_1)^2- (B^0_1)^2$ with the half-ellipse $\nu^2\la^2+y^2/(H^0_1)^4=1,\, y\ge0$. Considering all possible cases gives the proof of the lemma.
\end{proof}
We choose $\la$ to be one of the values found in Lemma \ref{lemma1} such that $|\la|<1/\nu$, that is satisfying \eqref{condlambda}. The solution of \eqref{systemjump} is then
\begin{equation}
\begin{array}{ll}\label{}
a=-k\hvpu, \qquad
\a= - \nu\la H^0_1 ik\hvpu   \qquad\mbox{if }  k>0\, .
\end{array}
\end{equation}

%%%%%%%%%%%%%%%%%%%%%%%%%%%%%%%%%%%%%%%%%
For $k<0$ we proceed in a similar way, solving an algebraic system for the unknowns $(b,\b,k\hvpu)$:
\begin{equation}
\begin{array}{ll}\label{systemjump2}
\begin{pmatrix}
-1&0 &1 \\
 0&1 &i \nu\la  H^0_1\\
 d& -i{\sla}H^0_1/{\nu \la}&0
\end{pmatrix}
\begin{pmatrix}
b \\
\b \\
k\hvpu
\end{pmatrix}=0\, .
\end{array}
\end{equation}
This system has a nontrivial solution under the same condition \eqref{condlamda2}.  The solution of \eqref{systemjump2} is then
\begin{equation}
\begin{array}{ll}\label{}
b=k\hvpu, \qquad
\b= - \nu\la H^0_1 ik\hvpu   \qquad\mbox{if }  k<0\, .
\end{array}
\end{equation}
Summarizing these results, we have shown that when $\la$ satisfies \eqref{condlamda2}, the solution of \eqref{equ1orderFourier}--\eqref{bc1orderFourier}, \eqref{decay2}, \eqref{decay3} is given by
\begin{equation}
\begin{array}{ll}\label{sol12}
\hat{U}^{(1)}(k,\eta,\tau)=\begin{cases}
-|k|\hvpu(k,\tau)e^{-k\eta}\Rb, &\qquad\mbox{if }  k>0\, ,\\
-|k|\hvpu(k,\tau)e^{k\eta}\overline{\Rb},&\qquad\mbox{if }  k<0\, ,
\end{cases}
\end{array}
\end{equation}
%%%%%%%%%%%%%%%%%%%%%%%%%%%%%%%%%%%%%%%%%
\begin{equation}
\begin{array}{ll}\label{sol22}
\hat{V}^{(1)}(k,\eta,\tau)=
  H^0_1 \hvpu(k,\tau)e^{\sla k\eta}\begin{pmatrix}
- {\sla}|k| \\
{i}k\\
-i \nu\la k\end{pmatrix}  .

\end{array}
\end{equation}
This solution depends on the unknown function $\hvpu(k,\tau)$, which describes the profile of the surface wave. By imposing solvability conditions on the equations for the second order corrections to this first order solution \eqref{sol12}, \eqref{sol22}, we will derive an evolution equation for the function $\hvpu(k,\tau)$.
%%%%%%%%%%%%%%%%%%%%%%%%%%%%%%%%%%%%%%%%%%%%%%%%%%%%%%%

%%%%%%%%%%%%%%%%%%%%%%%%%%%%%%%%%%%%%%%%%%%%%%%%%%%%%%%
\section{The second order equations}\label{second order}

Introducing the Fourier transforms
\[
\hat{U}^{(2)}(k,\eta,\tau)=\frac1{2\pi}\intR U^{(2)}(\th,\eta,\tau)e^{-ik\th}d\th,
\]
\[
\hat{V}^{(2)}(k,\eta,\tau)=\frac1{2\pi}\intR V^{(2)}(\th,\eta,\tau)e^{-ik\th}d\th,
\]
\[
\hat{\vphi}^{(2)}(k,\tau)=\frac1{2\pi}\intR \vphi^{(2)}(\th,\tau)e^{-ik\th}d\th,
\]
and Fourier transforming \eqref{equ2order}--\eqref{bc2order} with respect to $\th$, we find the equations
%%%%%%%%%%%%%%%%%%%%%%%%%%%%%%%
\begin{equation}
\begin{cases}\label{equ2orderFourier}
(\la-v^0_1)ik\hat v^{(2)} + ik B^0_1\hat B^{(2)}-\begin{pmatrix}
ik \\
 \p_\eta
\end{pmatrix}\hat q^{(2)}=\hat p_1\, ,\\
(\la-v^0_1)ik \hat B^{(2)} + ik B^0_1\hat v^{(2)}=\hat p_2\, ,\\
-ik\hat v_1^{(2)}-\p_\eta\hat v_2^{(2)}=\hat p_3\, ,&\qquad\mbox{for }  \eta>0\, ,\\
\end{cases}
\end{equation}
%%%%%%%%%%%%
\begin{equation}
\begin{cases}\label{equ22orderFourier}
\nu\la ik\hat H_1^{(2)}-\p_\eta\hat E^{(2)}=\hat p_1'\, ,\\
\nu\la ik\hat H_2^{(2)}+ik\hat E^{(2)}=\hat p_2'\, ,\\
\nu\la ik\hat E^{(2)}+ik\hat H_2^{(2)}-\p_\eta\hat H_1^{(2)}=\hat p_3'\, ,&\qquad\mbox{for }  \eta<0\, ,
\end{cases}
\end{equation}
%%%%%%%%
\begin{equation}
\begin{cases}\label{bc2orderFourier}
\ds (\la-v^0_1)ik \hvpd + \hat v_2^{(2)}=\hat r_1\, ,\\
\ds ikB^0_1 \hvpd - \hat B_2^{(2)}=\hat r_2\, , \qquad
ikH^0_1 \hvpd - \hat H_2^{(2)}=\hat r_3\, ,\\
\ds \hat q^{(2)}-H_1^0\hat H_1^{(2)}=\hat r_4\, , \qquad
\hat E^{(2)}+i\nu\la kH_1^0 \hvpd=\hat r_5 \, ,&\qquad\mbox{for }  \eta=0\, .
\end{cases}
\end{equation}
%%%%%%%%%%%%%%%%%%%%%%%%%%%%%%%%%%%%%%%%%%%%%%%%%%%%%%%

\subsection{The second order equations in the plasma region}

Let us first consider problem \eqref{equ2orderFourier}, that we write in the form
\begin{equation}
\begin{array}{ll}\label{equ2orderFourier2}
ik\Ac \hat{U}^{(2)}+\Bc\p_\eta \hat{U}^{(2)}=\hat p\, .
\end{array}
\end{equation}
From \eqref{decay1}, the solution of \eqref{equ2orderFourier2} must satisfy the decay condition
\begin{equation}\label{decay4}
\ds \lim_{\eta\to+\infty}\hat{U}^{(2)}(k,\eta,\tau)=0\, .
\end{equation}
In order to solve \eqref{equ2orderFourier2}, \eqref{decay4}, as in \cite{ali-hunter} we introduce a left eigenvector $\Lb$ such that
\[
\Lb\cdot (i\Ac-\Bc)=0\, ,
\]
normalized by
\begin{equation}
\label{Lnorm}
\Lb\cdot\Bc\Rb=\overline\Lb\cdot\Bc\overline\Rb=1\, .
\end{equation}
It follows from the equations satisfied by $\Lb,\Rb$ that
\begin{equation}
\label{equL}
\Lb\cdot\Bc\overline\Rb=\overline\Lb\cdot\Bc\Rb=0\, .
\end{equation}
We compute $\Lb$ and obtain
\[
\Lb=-\frac1{2id(\la-v^0_1)}\Rb\, .
\]
We also introduce a linear subspace consisting of the vectors $\Sb$ such that
\begin{equation}\label{equS}
\Lb\cdot \Bc\Sb=\overline\Lb \cdot \Bc \Sb=0\, .
\end{equation}
This subspace is complementary to the subspace spanned by $\{\Rb,\overline\Rb\}$.
We look for a solution of \eqref{equ2orderFourier2} in the form
\begin{equation}
\begin{array}{ll}\label{formU2}
\hat{U}^{(2)}(k,\eta,\tau)=\Sb(k,\eta,\tau)+a(k,\eta,\tau)\Rb+b(k,\eta,\tau)\overline\Rb,
\end{array}
\end{equation}
where $\Sb$ satisfies \eqref{equS}. We will solve for the vector-valued function $\Sb$ and the scalar functions $a,b$.
Substituting \eqref{formU2} in \eqref{equ2orderFourier2} gives
\begin{equation}
\begin{array}{ll}\label{46}
ik\Ac\Sb+\Bc\p_\eta\Sb+(\p_\eta a+ka)\Bc\Rb+(\p_\eta b-kb)\Bc\overline\Rb=\hat p.
\end{array}
\end{equation}
Left multiplying \eqref{46} by $\Lb$ and $\overline\Lb$, and using \eqref{equL}, \eqref{equS}, we find the equations
\begin{equation*}
\begin{array}{ll}\label{}
\p_\eta a+ka=\Lb\cdot \hat p\, ,\qquad  \p_\eta b-kb=\overline\Lb\cdot \hat p\, ,
\end{array}
\end{equation*}
whose solutions are given by
\begin{eqnarray}
a(k,\eta,\tau)=e^{-k\eta} \left( a_0(k,\tau)+ \int_0^\eta \Lb\cdot \hat p(k,\eta',\tau)e^{k\eta'}d\eta'\right),
 \label{equa} \\
b(k,\eta,\tau)=e^{k\eta} \left( b_0(k,\tau)+ \int_0^\eta \overline\Lb\cdot \hat p(k,\eta',\tau)e^{-k\eta'}d\eta'\right),
\label{equb} 
\end{eqnarray}  
where $a_0(k,\tau),b_0(k,\tau)$ are arbitrary functions of integration, that will be chosen later.

%%%%%%%%%%%%%%%%%%%%%%%%%%%%%%%
Next, we solve \eqref{46} for $\Sb$. From \eqref{equS}, vectors $\Sb$  of the above linear subspace have the form
\[
\Sb=(S_1,0,S_3,S_4,0)^T, 
\]
with arbitrary components $S_1,S_3,S_4$. We introduce vectors $\Lb_j$, with $j=1,3,4$, such that 
\begin{equation}
\begin{array}{ll}\label{equLj}
\Lb_j\cdot \Bc\Rb =\Lb_j\cdot \Bc\overline\Rb = 0, \qquad i\Lb_j\cdot \Ac\Sb = S_j.
\end{array}
\end{equation}
They are given explicitly by
\[
\Lb_1= \frac{1}{d}\left({-i(\la-v^0_1)},0 , {iB^0_1},0, 0\right)^T\, ,
\qquad
\Lb_3=  \frac{1}{d}\left({iB^0_1},0 , -{i(\la-v^0_1)},0, 0\right)^T\, ,
\]
\[
\Lb_4= \left(0,0 , 0, - \frac{i}{\la-v^0_1}, 0\right)^T\, .
\]
Left multiplying \eqref{46} by $\Lb_j$ and using \eqref{equLj} gives
\[ S_j=\frac1{k}\Lb_j\cdot \hat p, \qquad\mbox{for  }j=1,3,4 \, .
\]
Thus the solution for $\Sb$ is given by\footnote{The simpler form of $\Sb$ in \eqref{solS}, with respect to (6.16) in \cite{ali-hunter}, seems due to the choice of the symmetric form of equations \eqref{mhd2}, instead of the conservative form \eqref{mhd1}.}
\begin{equation}
\begin{array}{ll}\label{solS}
\ds \Sb=\left(\frac1{k}\Lb_1\cdot \hat p, 0, \frac1{k}\Lb_3\cdot \hat p, \frac1{k}\Lb_4\cdot \hat p, 0 \right)^T \, .
\end{array}
\end{equation}
%%%%%%%%%%%%%%%%%%%%%%%%%%%%%%%
%%%%%%%%%%%%%%%%%%%%%%%%%%%%%%%
%%%%%%%%%%%%%%%%%%%%%%%%%%%%%%%
We compute the Fourier transform of the right-hand sides of \eqref{equ2orderFourier}. For $p_1=(p_{11}, p_{12})$ we have
\begin{multline*}
\hat p_{11}(k,\eta,\tau)= -(\la -v^0_1)|k| e^{-|k|\eta} \hvpu_\tau(k,\tau)\\
-id\intR |k-\ell | \, \ell \, (|k-\ell | - | \ell| ) e^{-(|k-\ell | + | \ell| )\eta} \hvpu(k-\ell) \hvpu(\ell)\, d\ell\, ,
\end{multline*}
%%%
\begin{multline*}
\hat p_{12}(k,\eta,\tau)= i(\la -v^0_1)k e^{-|k|\eta} \hvpu_\tau(k,\tau)\\
+d\intR (k-\ell ) \, \ell \,  | \ell| \left(  e^{-(|k-\ell | + | \ell| )\eta}  - e^{- | \ell| \eta}  \right) \hvpu(k-\ell) \hvpu(\ell)\, d\ell \\
-d\intR |k-\ell | \, \ell^2 e^{-(|k-\ell | + | \ell| )\eta} \hvpu(k-\ell) \hvpu(\ell)\, d\ell \, ,
\end{multline*}
disregarding in the integrals the dependence on $\tau$, for the sake of simplicity.
%%%%
For $p_2=(p_{21}, p_{22})$ and $p_3$ we obtain
\[
\hat p_{21}(k,\eta,\tau)=  B^0_1|k| e^{-|k|\eta} \hvpu_\tau(k,\tau) \, ,
\]
%%%
\[
\hat p_{22}(k,\eta,\tau)=  iB^0_1k e^{-|k|\eta} \hvpu_\tau(k,\tau) \, ,
\]
\[
\hat p_{3}(k,\eta,\tau)=  -i(\la -v^0_1)\intR (k-\ell ) \, \ell^2 e^{- | \ell| \eta} \hvpu(k-\ell) \hvpu(\ell)\, d\ell \, .
\]
%%%%%%%%%%%%%%%%%%%%%%%%%%%%%%%%%%%%%%%%%%%%%%%%%%%%%%%
It follows that
\begin{multline}\label{Lhp}
\Lb\cdot \hat p(k,\eta,\tau)= i\frac{  (\la -v^0_1)^2+(B^0_1)^2 }{2d(\la -v^0_1)} \, (k-|k|)e^{-|k|\eta} \hvpu_\tau(k,\tau)\\
+\frac1{2} \intR \ell \left\{ |k-\ell |  (|k-\ell | - | \ell| ) + |k-\ell |  \ell - (k-\ell )  | \ell | 
\right\}
 e^{-(|k-\ell | + | \ell| )\eta} \hvpu(k-\ell) \hvpu(\ell)\, d\ell \\
+\frac1{2}  \intR (k-\ell ) \,| \ell | (|\ell | +  \ell )e^{- | \ell | \eta} \hvpu(k-\ell) \hvpu(\ell)\, d\ell \, ,
\end{multline}
%%%
\begin{multline}\label{Lhpbar}
\overline\Lb\cdot \hat p(k,\eta,\tau)= i\frac{  (\la -v^0_1)^2+(B^0_1)^2 }{2d(\la -v^0_1)} \,  (k + |k|)e^{-|k|\eta} \hvpu_\tau(k,\tau)\\
+\frac1{2} \intR \ell \left\{ -|k-\ell |  (|k-\ell | - | \ell| ) + |k-\ell |  \ell - (k-\ell )  | \ell | 
\right\}
 e^{-(|k-\ell | + | \ell| )\eta} \hvpu(k-\ell) \hvpu(\ell)\, d\ell \\
+\frac1{2}  \intR (k-\ell ) \,| \ell | (  \ell - |\ell | )e^{- | \ell | \eta} \hvpu(k-\ell) \hvpu(\ell)\, d\ell \, ,
\end{multline}
%%%
and
\begin{equation}
\begin{array}{ll}\label{L1hp}

%\begin{multline}\label{L1hp}
\ds \Lb_1\cdot \hat p(k,\eta,\tau)= i\frac{  (\la -v^0_1)^2+(B^0_1)^2 }{d} \,  |k|e^{-|k|\eta} \hvpu_\tau(k,\tau)\\
 \qquad   -(\la -v^0_1) \intR |k-\ell | \, \ell \, (|k-\ell | - | \ell| )  e^{-(|k-\ell | + | \ell| )\eta} \hvpu(k-\ell) \hvpu(\ell)\, d\ell 
 \, , \\
 \\
%\end{multline}
%%%
%\begin{multline}\label{}
\ds \Lb_3\cdot \hat p(k,\eta,\tau)= - \frac{ 2iB^0_1 (\la -v^0_1)}{d} \, |k|e^{-|k|\eta} \hvpu_\tau(k,\tau)\\
 \qquad   +B^0_1 \intR |k-\ell | \, \ell \, (|k-\ell | - | \ell| )  e^{-(|k-\ell | + | \ell| )\eta} \hvpu(k-\ell) \hvpu(\ell)\, d\ell 
 \, , \\
 \\
%\end{multline}
%%%
%\begin{equation}\label{L4hp}
\ds \Lb_4\cdot \hat p(k,\eta,\tau)= \frac{ B^0_1}{\la -v^0_1} \, ke^{-|k|\eta} \hvpu_\tau(k,\tau)
 \, .
\end{array}
\end{equation}
The expressions obtained in \eqref{Lhp}--\eqref{L1hp} are to be inserted in \eqref{equa}, \eqref{equb}, \eqref{solS} to give $a,b,\Sb$.

In order to verify the decay condition \eqref{decay4} for $\hat{U}^{(2)}(k,\eta,\tau)$, given by \eqref{formU2}, we first notice that $\Sb(k,\eta,\tau)$ depends on $\eta$ only through the exponentials of $-|k|\eta$ and $-(|k-\ell | + | \ell| )\eta$, see \eqref{solS} and \eqref{L1hp}, so that 
\begin{equation*}\label{}
\ds \lim_{\eta\to+\infty}\Sb(k,\eta,\tau)=0\, .
\end{equation*}
Thus $\hat{U}^{(2)}(k,\eta,\tau)$ satisfies \eqref{decay4} if and only if 
\begin{equation}\label{decaya}
\ds \lim_{\eta\to+\infty}a(k,\eta,\tau)=0\, ,
\end{equation}
\begin{equation}\label{decayb}
\ds  \lim_{\eta\to+\infty}b(k,\eta,\tau)=0\, .
\end{equation}
From \eqref{equa}, \eqref{equb}, \eqref{Lhp}, \eqref{Lhpbar}, condition \eqref{decaya} is automatically satisfied if $k>0$, and \eqref{decayb} is automatically satisfied if $k<0$. It follows that $a_0$ remains undetermined for $k>0$, and $b_0$ remains undetermined for $k<0$.
Instead, \eqref{decaya}, \eqref{decayb} may be used to determine $a_0$ if $k<0$, and $b_0$ if $k>0$, as functions of $\hvpu$ through \eqref{Lhp}, \eqref{Lhpbar}:
\begin{multline}\label{deta0}
\ds   a_0(k,\tau)= - \int_0^{+\infty} \Lb\cdot \hat p(k,\eta',\tau)e^{k\eta'}d\eta' 
\ds  = i\, \frac{  (\la -v^0_1)^2+(B^0_1)^2 }{2d(\la -v^0_1)}  \hvpu_\tau(k,\tau)\\
\ds -\frac1{2} \intR \ell \, \frac{  |k-\ell |\ell-  (k-\ell ) | \ell|  + |k-\ell |  (|k-\ell | -  | \ell |) 
}{|k-\ell | +|k|+  | \ell |}\,
 \hvpu(k-\ell) \hvpu(\ell)\, d\ell \\
\ds -\frac1{2}  \intR \frac{  (k-\ell )| \ell|    (   | \ell | + \ell  ) 
}{  |k| + | \ell |}\,
 \hvpu(k-\ell) \hvpu(\ell)\, d\ell \, , \qquad \mbox{if  }k<0\, ,
\end{multline}
%%%
\begin{multline}\label{detb0}
\ds b_0(k,\tau)= - \int_0^{+\infty} \overline\Lb\cdot \hat p(k,\eta',\tau)e^{-k\eta'}d\eta' 
\ds  =- i\, \frac{  (\la -v^0_1)^2+(B^0_1)^2 }{2d(\la -v^0_1)}  \hvpu_\tau(k,\tau)\\
\ds -\frac1{2} \intR \ell \, \frac{  |k-\ell |\ell-  (k-\ell ) | \ell|  - |k-\ell |  (|k-\ell | -  | \ell |) 
}{|k-\ell | +|k| +  | \ell |}\,
 \hvpu(k-\ell) \hvpu(\ell)\, d\ell \\
\ds +\frac1{2}  \intR \frac{  (k-\ell )| \ell|    (   | \ell | -\ell  ) 
}{  |k| +  | \ell |}\,
 \hvpu(k-\ell) \hvpu(\ell)\, d\ell \, , \qquad \mbox{if  }k>0\, .
\end{multline}

%%%%%%%%%%%%%%%%%%%%%%%%%%%%%%%%%%%%%%%%%%%%%%%%%%%%%%%

\subsection{The second order equations in vacuum}

Let us consider problem \eqref{equ22orderFourier} for $\eta<0$. From the second equation in \eqref{equ22orderFourier}, 
$
ik \hat H_2^{(2)}=(\hat p'_2- ik\hat E^{(2)})/\nu\la$, and substituting in the other equations of \eqref{equ22orderFourier} we get
\begin{equation}
\begin{array}{ll}\label{equE2}
 \p_\eta^2 \hat E^{(2)}+k^2(\nu^2\la^2-1)\hat E^{(2)}= - P,\end{array}
\end{equation}
where
\begin{equation}
\begin{array}{ll}\label{defP}
P=  \nu\la ik \hat p'_3 - ik\hat p'_2 + \p_\eta \hat p'_1
\, .
\end{array}
\end{equation}
We solve \eqref{equE2} with the decay condition
\begin{equation}\label{decay5}
\ds \lim_{\eta\to-\infty}\hat{E}^{(2)}(k,\eta,\tau)=0\, 
\end{equation}
(obtained from \eqref{decay1}), and \eqref{condlambda}.
The general solution of \eqref{equE2} is
\begin{equation}\label{genersol3}
\hat{E}^{(1)}(k,\eta,\tau)=
\a'(k,\tau)e^{\sla k\eta} + \b'(k,\tau)e^{-\sla k\eta}
+\frac1{2|k|\sla}\intRneg e^{-\sla |k| | \eta-\zeta|} P(k,\zeta,\tau)\, d\zeta
\, ,
\end{equation}
where $\a'(k,\tau)$ and $\b'(k,\tau)$ are arbitrary complex-valued functions. From \eqref{equ22orderFourier}, \eqref{genersol3} the general solution for the other unknowns is
\begin{multline}\label{genersol4}
\hat{H}_1^{(1)}(k,\eta,\tau)=
\frac{\sla}{i\nu\la }\a'(k,\tau)e^{\sla k\eta} -   \frac{\sla}{i\nu\la } \b'(k,\tau)e^{-\sla k\eta}
\\
+\frac1{2\nu\la ik|k|\sla} \left\{ \intRneg e^{-\sla |k| | \eta-\zeta|} \p_\zeta P(k,\zeta,\tau)\, d\zeta - e^{-\sla |k \eta|}  P(k,0,\tau)\right\}
+\frac1{\nu\la ik}\hat p'_1(k,\eta,\tau) \, ,
\end{multline}
%%%
%%%
\begin{multline}\label{genersol5}
\hat{H}_2^{(1)}(k,\eta,\tau)=
-\frac{1}{\nu\la }\a'(k,\tau)e^{\sla k\eta} -   \frac{1}{\nu\la } \b'(k,\tau)e^{-\sla k\eta}
\\
-\frac1{2\nu\la |k|\sla}\intRneg e^{-\sla |k| | \eta-\zeta|}  P(k,\zeta,\tau)\, d\zeta
+\frac1{\nu\la ik}\hat p'_2(k,\eta,\tau) 
\, .
\end{multline}
Imposing the decay condition
\begin{equation*}\label{decay6}
\ds \lim_{\eta\to-\infty}\hat{V}^{(2)}(k,\eta,\tau)=0\, 
\end{equation*}
to \eqref{genersol3}--\eqref{genersol5} yields that the solution of \eqref{equ22orderFourier} is given by
%%%%%%%%%%%%%%%%%%%%%%%%%%%%%%%%%%%%%%%%%
\begin{multline}\label{sol221}
\hat{V}^{(2)}(k,\eta,\tau)=
 \a'(k,\tau)e^{\sla k\eta}
 \begin{pmatrix}
\frac{\sla}{i\nu\la } \\
-\frac{1}{\nu\la}\\
1
\end{pmatrix} 
\\
+  \frac1{2|k|\sla}\begin{pmatrix}
\frac1{\nu\la ik} \left\{ \intRneg e^{-\sla |k| | \eta-\zeta|} \p_\zeta P(k,\zeta,\tau)\, d\zeta - e^{-\sla |k \eta|}  P(k,0,\tau) \right\}\\
-\frac1{\nu\la }\intRneg e^{-\sla |k| | \eta-\zeta|}  P(k,\zeta,\tau)\, d\zeta\\
\intRneg e^{-\sla |k| | \eta-\zeta|} P(k,\zeta,\tau)\, d\zeta
\end{pmatrix}
+  \begin{pmatrix}
\frac1{\nu\la ik}\hat p'_1 \\
\frac1{\nu\la ik}\hat p'_2\\
0
\end{pmatrix}
, \qquad\mbox{if }  k>0\, ,
\end{multline}
%%%%%%%%%%%%
\begin{multline}\label{sol222}
\hat{V}^{(2)}(k,\eta,\tau)=
 \b'(k,\tau)e^{-\sla k\eta}
 \begin{pmatrix}
-\frac{\sla}{i\nu\la } \\
-\frac{1}{\nu\la}\\
1
\end{pmatrix} 
\\
+  \frac1{2|k|\sla}\begin{pmatrix}
\frac1{\nu\la ik} \left\{ \intRneg e^{-\sla |k| | \eta-\zeta|} \p_\zeta P(k,\zeta,\tau)\, d\zeta - e^{-\sla |k \eta|}  P(k,0,\tau) \right\} \\
-\frac1{\nu\la }\intRneg e^{-\sla |k| | \eta-\zeta|}  P(k,\zeta,\tau)\, d\zeta\\
\intRneg e^{-\sla |k| | \eta-\zeta|} P(k,\zeta,\tau)\, d\zeta
\end{pmatrix}
+  \begin{pmatrix}
\frac1{\nu\la ik}\hat p'_1 \\
\frac1{\nu\la ik}\hat p'_2\\
0
\end{pmatrix}
, \qquad\mbox{if }  k<0\, .
\end{multline}
Notice that we need to determine the arbitrary functions $ \a'(k,\tau)$ if $k>0$, and $ \b'(k,\tau)$ if $k<0$.

Substituting \eqref{sol22} in the right-hand sides of \eqref{equ22orderFourier} gives
\begin{equation}
\begin{array}{ll}\label{p'1}

\hat p'_1(k,\eta,\tau)= - \nu\sla H^0_1 |k| e^{\sla |k|\eta} \, \hvpu_\tau(k,\tau)
\\
\ds \qquad  \qquad  +i\nu\la (\nu^2\la^2-1)H^0_1\intR(k-\ell )\, \ell^2 e^{\sla |\ell | \eta}  \hvpu(k-\ell) \hvpu(\ell)\, d\ell,
\\
\\
%\end{multline}
%%%
%\begin{equation}
%\begin{array}{ll}\label{p'2}
\hat p'_2(k,\eta,\tau)= i\nu H^0_1 k e^{\sla |k| \eta} \, \hvpu_\tau(k,\tau),
\\
\\
%\end{array}
%\end{equation}
%%%
%\begin{multline*}
\hat p'_3(k,\eta,\tau)= - i\nu^2\la  H^0_1 k e^{\sla |k| \eta} \, \hvpu_\tau(k,\tau)
\\
\ds \qquad  \qquad  + (\nu^2\la^2-1)\sla H^0_1 \intR(k-\ell )\, \ell |\ell | e^{\sla |\ell | \eta}  \hvpu(k-\ell) \hvpu(\ell)\, d\ell,
\end{array}
\end{equation}
and from here we also obtain
\begin{multline}\label{P}
P(k,\zeta,\tau)= 2\nu^3\la^2 H^0_1 k^2 e^{\sla |k| \zeta} \, \hvpu_\tau(k,\tau)
\\
+i \nu\la  (\nu^2\la^2-1)\sla H^0_1 \intR (k^2-\ell^2 )\, \ell |\ell | e^{\sla |\ell | \zeta}  \hvpu(k-\ell) \hvpu(\ell) \, d\ell,
\end{multline}
%%%%%%
\begin{multline}\label{detaP}
\p_\zeta P(k,\zeta,\tau)= 2\nu^3\la^2 \sla H^0_1 k^2 |k| e^{\sla |k| \zeta} \, \hvpu_\tau(k,\tau)
\\
-i \nu\la  (\nu^2\la^2-1)^2 H^0_1 \intR (k^2-\ell^2 )\, \ell^3 e^{\sla |\ell | \zeta}  \hvpu(k-\ell) \hvpu(\ell) \, d\ell.
\end{multline}
Then we substitute \eqref{p'1}--\eqref{detaP} in \eqref{sol221}, \eqref{sol222}.
%%%%%%%%%%%%%%%%%%%%%%%%%%%%%%%%%%%%%%%%%%%%%%%%%%%%%%%

\subsection{The second order jump conditions}

The first-order solution depends on the unknown function $\hvpu(k,\tau)$, which describes the profile of the surface wave, while the second order solution depends, in addition, on unknown functions $a_0(k,\tau), b_0(k,\tau)$ and $\a'(k,\tau), \b'(k,\tau)$. In this section we study the second order jump conditions. We show that they reduce to a singular linear system of algebraic equations for $(a_0,b_0,\a',\b', \hvpd)$, where $\hvpd(k,\tau)$ is the Fourier transform of the second-order displacement of the interface. 
Imposing solvability conditions on this system gives the evolution equation for the function $\hvpu(k,\tau)$ that we seek.

Let us consider the jump conditions \eqref{bc2orderFourier}, where we substitute the second order corrections obtained in the previous sections.

Let us first assume $k>0$, recalling that in this case we need to determine $a_0(k,\tau),\a'(k,\tau)$ and $ \hvpd(k,\tau)$ (for $k>0$). From \eqref{bc2orderFourier}, \eqref{formU2}, \eqref{equa}, \eqref{solS}, \eqref{L1hp} and \eqref{sol221}, \eqref{p'1}--\eqref{detaP},
evaluated at $\eta=0$, we obtain the linear system
\begin{equation}
\begin{array}{ll}\label{algebra1}
\begin{pmatrix}
1 & 0 &1\\
1 & 0 & 1 \\
0 & 1/\nu\la & iH^0_1\\
d & i\sla H^0_1/\nu\la &  0 \\
0& 1 & i\nu\la H^0_1
\end{pmatrix}
\begin{pmatrix}
a_0\\
\a' \\
k\hvpd
\end{pmatrix}
=
\begin{pmatrix}
\hat r'_1 \\
\hat r'_2 \\
\hat r'_3\\
\hat r'_4 \\
\hat r'_5
\end{pmatrix},

\end{array}
\end{equation}
where we have set
\begin{equation*}
\begin{array}{ll}\label{}
\ds \hat r'_1= \frac1{i(\la-v^0_1)}\hat r_1 + b_0 \, ,\\
\ds \hat r'_2 = \frac1{iB^0_1}\left(  \hat r_2 +  \frac{B^0_1}{\la-v^0_1}\hvpu_\tau \right)+ b_0\, ,\\
\ds \hat r'_3=\hat r_3 +  \frac{H^0_1}{\la}\hvpu_\tau
-\frac1{2\nu\la |k|\sla}\intRneg e^{\sla |k|  \zeta}  P(k,\zeta,\tau)\, d\zeta
\, ,
\\
\ds \hat r'_4=\hat r_4 -b_0d - (H^0_1)^2 \frac{\sla }{\la i}\, \frac{k}{|k|} \hvpu_\tau
 +\frac{H^0_1}{2i\nu\la  \sla} \, \frac{1}{k|k|} \left\{ \intRneg e^{\sla |k|  \zeta}  \p_\zeta P(k,\zeta,\tau)\, d\zeta -  P(k,0,\tau)\right\}
\\
\ds \qquad +\frac{(H^0_1)^2   (\nu^2\la^2-1) }{k} \intR (k-\ell )\, \ell^2  \hvpu(k-\ell) \hvpu(\ell) \, d\ell \, ,
\\
\ds \hat r'_5= \hat r_5
-\frac1{2 |k|\sla}\intRneg e^{\sla |k|  \zeta}  P(k,\zeta,\tau)\, d\zeta
\, ,

\end{array}
\end{equation*}
with $b_0$ given by \eqref{detb0}.
%%%
First of all we see that the first two lines of the matrix in the left-hand side of \eqref{algebra1} are equal, and we can verify that $\hat r'_1 =\hat r'_2 $. Moreover, the last row of the matrix in \eqref{algebra1} equals the third one multiplied by $\nu\la$, and actually one verifies that $\hat r'_5 =\nu\la \hat r'_3 $. Thus \eqref{algebra1} may be reduced to 
\begin{equation}
\begin{array}{ll}\label{algebra3}
\begin{pmatrix}
1 & 0 &1\\
0 & 1/\nu\la & iH^0_1\\
d & i\sla H^0_1/\nu\la &  0 \\
\end{pmatrix}
\begin{pmatrix}
a_0\\
\a' \\
k\hvpd
\end{pmatrix}
=
\begin{pmatrix}
\hat r'_1 \\
\hat r'_3\\
\hat r'_4 \\
\end{pmatrix},

\end{array}
\end{equation}
The determinant of the matrix of this system is zero because of \eqref{condlamda2}, i.e. the equation defining $\la$. It is easily seen that the rank of this matrix is 2. Then, the linear system \eqref{algebra3} is solvable if and only if the rank of the augmented matrix is also equal to 2, and this is true if the following condition holds:
\begin{equation}
\begin{array}{ll}\label{solvability0}
d\hat r'_3+iH^0_1 \hat r'_4 -iH^0_1d\hat r'_1=0\, .
\end{array}
\end{equation}
Developing the terms in \eqref{solvability0} we get the solvability condition
\begin{equation}
\begin{array}{ll}\label{solvability1}
\ds   \left( 2 \frac{\la -v^0_1}{d} + \frac{\nu^2\la}{\sla^2} \right)\, \hvpu_\tau(k,\tau) + i \intR \La_+(k,\ell )\hvpu(k-\ell, \tau) \hvpu(\ell, \tau) \, d\ell 
=0\, , \qquad k>0\, ,
\end{array}
\end{equation}
where we have denoted
%%%%%%%%%%%%%%%%%%%%%%%%%%%%%%
\begin{multline}\label{Lapiu}
\ds \La_+(k,\ell ) = \ell \frac{|k-\ell | (|k-\ell| -|\ell | ) + (k-\ell ) |\ell | -|k-\ell | \ell }{|k-\ell |  +|k| +|\ell |  } +
\frac{ (k-\ell )  |\ell |(|\ell | - \ell  ) } {|k| +|\ell |  } - (k-\ell ) |\ell | \\
+ \sla 
\Big\{ -k |\ell |
+\frac1{2} \big( (k +\ell  ) \ell - |k-\ell | |\ell |  \big) \Big\} \, .
\end{multline}
Thus, when $k>0$, the system \eqref{algebra3} is solvable if and only if $\hvpu$ satisfies equation \eqref{solvability1} and then the rank of the augmented matrix of the system is equal to 2.
Given the solution $\hvpu$ of \eqref{solvability1} we compute $\hat U^{(1)}, \hat V^{(1)}$ from \eqref{sol12}, \eqref{sol22}. Thus the leading-order term of the asymptotic expansion is uniquely determined. From system \eqref{algebra3} we may obtain $a_0,\a'$ in terms of an arbitrary second order wave profile $\hvpd$, and in turn $\hat U^{(2)}, \hat V^{(2)}$ from \eqref{formU2}, \eqref{equa}, \eqref{solS}, \eqref{L1hp} and \eqref{sol221}, \eqref{p'1}--\eqref{detaP}. The wave profile $\hvpd$ should be determined by considering higher order terms of the asymptotic expansion, see \cite{marcou}.

The case $k<0$ is similar.  Now we need to determine $b_0(k,\tau),\b'(k,\tau)$ and $ \hvpd(k,\tau)$ (for $k<0$). From \eqref{bc2orderFourier}, \eqref{formU2}, \eqref{equb}, \eqref{solS}, \eqref{L1hp} and \eqref{sol222}--\eqref{detaP},
evaluated at $\eta=0$, we obtain the linear system
\begin{equation}
\begin{array}{ll}\label{algebra4}
\begin{pmatrix}
-1 & 0 &1\\
-1 & 0 & 1 \\
0 & 1/\nu\la & iH^0_1\\
d & -i\sla H^0_1/\nu\la &  0 \\
0& 1 & i\nu\la H^0_1
\end{pmatrix}
\begin{pmatrix}
b_0\\
\b' \\
k\hvpd
\end{pmatrix}
=
\begin{pmatrix}
\hat r''_1 \\
\hat r''_2 \\
\hat r'_3\\
\hat r''_4 \\
\hat r'_5
\end{pmatrix},

\end{array}
\end{equation}
%%%
where we have set
\begin{equation*}
\begin{array}{ll}\label{}
\ds \hat r''_1= \frac1{i(\la-v^0_1)}\hat r_1 - a_0 \, ,\\
\ds \hat r''_2 = \frac1{iB^0_1}\left(  \hat r_2 +  \frac{B^0_1}{\la-v^0_1}\hvpu_\tau \right) - a_0\, ,\\
\ds \hat r''_4=\hat r_4 -a_0d - (H^0_1)^2 \frac{\sla }{\la i}\, \frac{k}{|k|} \hvpu_\tau
 +\frac{H^0_1}{2i\nu\la  \sla} \, \frac{1}{k|k|} \left\{ \intRneg e^{\sla |k|  \zeta}  \p_\zeta P(k,\zeta,\tau)\, d\zeta - P(k,0,\tau) \right\}
\\
\ds \qquad +\frac{(H^0_1)^2   (\nu^2\la^2-1) }{k} \intR (k-\ell )\, \ell^2  \hvpu(k-\ell) \hvpu(\ell) \, d\ell \, ,
\end{array}
\end{equation*}
with $a_0$ given by \eqref{deta0}.
%%%
By similar arguments as before we show that the linear system \eqref{algebra4} is solvable if and only if  
\begin{equation*}
\begin{array}{ll}\label{}
d\hat r'_3-iH^0_1 \hat r''_4 -iH^0_1d\hat r''_1=0\, .
\end{array}
\end{equation*}
Expanding the terms we get the solvability condition
\begin{equation}
\begin{array}{ll}\label{solvability2}
\ds    \left( 2 \frac{\la -v^0_1}{d} + \frac{\nu^2\la}{\sla^2} \right)\, \hvpu_\tau(k,\tau) + i \intR \La_-(k,\ell )\hvpu(k-\ell, \tau) \hvpu(\ell, \tau) \, d\ell 
=0\, , \qquad k<0\, ,
\end{array}
\end{equation}
where we have denoted
\begin{multline}\label{Lameno}
\ds \La_-(k,\ell ) = \ell \frac{|k-\ell | (|k-\ell| -|\ell | ) - (k-\ell ) |\ell | + |k-\ell | \ell }{|k-\ell |  +|k| +|\ell |  } +
\frac{ (k-\ell )  |\ell |(|\ell | + \ell  ) } {|k| +|\ell |  } - (k-\ell ) |\ell | \\
+ \sla 
\Big\{ -k |\ell |
-\frac1{2} \big( (k +\ell  ) \ell - |k-\ell | |\ell |  \big) \Big\} \, .
\end{multline}
Given the solution $\hvpu$ of \eqref{solvability2} we compute $\hat U^{(1)}, \hat V^{(1)}$ from \eqref{sol12}, \eqref{sol22}. From system \eqref{algebra4} we may get $b_0,\b'$ in terms of an arbitrary second order wave profile $\hvpd$, and in turn $\hat U^{(2)}, \hat V^{(2)}$ from \eqref{formU2}, \eqref{equb}, \eqref{solS}, \eqref{L1hp} and \eqref{sol222}--\eqref{detaP}.
Also for $k<0$ the wave profile $\hvpd$ should be determined by considering higher order terms of the asymptotic expansion, see \cite{marcou}.

%%%%%%%%%%%%%%%%%%%%%%%%%%%%%%%
%\vfill
%\eject

\subsection{The kernel}

The equations \eqref{solvability1}, \eqref{solvability2} can be written in more compact form as
\begin{equation}
\begin{array}{ll}\label{solvability3}
\ds  \left( 2 \frac{\la -v^0_1}{d} + \frac{\nu^2\la}{\sla^2} \right)\, \hvpu_\tau(k,\tau) + i \intR \La_0(k,\ell )\hvpu(k-\ell, \tau) \hvpu(\ell, \tau) \, d\ell 
=0\, , \qquad \forall \, k\not=0\, ,
\end{array}
\end{equation}
with
\begin{equation*}
\begin{array}{ll}\label{}
\La_0(k,\ell )=\La_{01}(k,\ell )+\La_{02}(k,\ell ),
\end{array}
\end{equation*}
%%%%%%%%%%%%%%%%%%%%%%%%%
\begin{equation*}
\begin{array}{ll}\label{}
\ds  \La_{01}(k,\ell )= \sgn(k) \Big\{  
\ell \frac{  (k-\ell ) |\ell | - |k-\ell | \ell }{|k-\ell |  +|k| +|\ell |  } -
\frac{ (k-\ell )  |\ell | \ell   } {|k| +|\ell |  }  
 + \frac\sla {2} 
 \big( (k +\ell  ) \ell - |k-\ell | |\ell |  \big) \Big\} \, ,

\end{array}
\end{equation*}
%%%%%%%%%%%%%%%%%%%%%%%%%
%%%%%%%%%%%%%%%%%%%%%%%%%

\begin{equation*}
\ds \La_{02}(k,\ell ) = \ell \frac{|k-\ell | (|k-\ell| -|\ell | )  }{|k-\ell |  +|k| +|\ell |  } +
\frac{ (k-\ell ) \ell^2 } {|k| +|\ell |  } - (k-\ell ) |\ell | 
- \sla 
k |\ell |
\, .
\end{equation*}
%%%%
The kernel $ \La_{01}(k,\ell )$ can also be written as
\[
 \La_{01}(k,\ell )=   \sgn(k) \, \tilde\La_{01}(k-\ell,\ell ),
\]
where
\begin{equation*}
\begin{array}{ll}\label{}
\ds  \tilde\La_{01}(k,\ell )= 
 \ell\, \frac{ k|\ell| -|k| \ell  }{|k |  +|k+\ell | +|\ell |  }   - \frac{ k\ell   |\ell | } {|k+\ell | +|\ell |  } + 
\frac{\sla}{2} \Big( (k +2 \ell ) \ell -|k\ell |\Big)
\, .
\end{array}
\end{equation*}
%%%%
On the other hand, the kernel $ \La_{02}(k,\ell )$ can also be written as
\[
 \La_{02}(k,\ell )=   \sgn(k) \, \tilde\La_{02}(k-\ell,\ell ),
\]
where
\begin{equation*}
\ds  \tilde\La_{02}(k,\ell )= \sgn(k+\ell ) \, \Big\{
 \, \frac{|k | \ell(|k| -|\ell |   ) }{|k+\ell  |  +|k| +|\ell |  } +
\frac{ k\ell^2 } {|k+\ell | +|\ell |  } -k|\ell | - \sla(k+\l)|\l|
\Big\}\, .
\end{equation*}
Moreover, the kernel $\tilde\La_{01}(k,\ell )+\tilde\La_{02}(k,\ell )$ can be equivalently replaced in the integral equation \eqref{solvability3} by the symmetrized kernel
\begin{equation}
\label{kernel2}
\tilde\La(k,\ell )=\frac1{2}\left(\tilde\La_{01}(k,\ell )+\tilde\La_{01}(\ell,k )+\tilde\La_{02}(k,\ell )+\tilde\La_{02}(\ell,k ) \right),
\end{equation}
because the antisymmetric part of $\tilde\La_{01}+\tilde\La_{02}$ gives a vanishing integral.
Thus we can write \eqref{solvability3} as
%%%%%%%%%%%%%
\begin{equation}
\begin{array}{ll}\label{solvability4}
\ds  \left( 2 \frac{\la -v^0_1}{d} + \frac{\nu^2\la}{\sla^2} \right)\,  \hvpu_\tau(k,\tau) + i\,\sgn(k) \intR \tilde\La(k-\ell,\ell )\,\hvpu(k-\ell, \tau)\, \hvpu(\ell, \tau) \, d\ell 
=0\, , \qquad \forall \, k\not=0\, ,
\end{array}
\end{equation}
where the kernel $\tilde\La$ in \eqref{solvability4} is explicitly given by
%%%%%%%%%%%%%%%%%%%%%%%%%%%
\begin{multline}\label{kernel}
\ds
\tilde\La(k,\ell)= \frac1{2}\Big\{  \frac{(k-\ell)(|k|\ell - k|\ell |)}{|k+\l|+|k|+|\l|} - \frac{k\ell |\l|}{|k+\l|+|l|} - \frac{|k|k\ell}{|k+\l|+|k|} 
 +
 \sla\Big( k^2+\l^2+k\l-|k\ell|
 \Big)
\Big\}
\\
\ds+ \frac1{2}\sgn(k+\l)\Big\{  \frac{(|k|-|\ell |)(|k|\ell - k|\ell |)}{|k+\l|+|k|+|\l|}  
 + \frac{k\ell^2}{|k+\l|+|\l|} + \frac{k^2\ell}{|k+\l|+|k|} -k |\l | - |k|\l - \sla (k+\l)( |k|+|\l|) 
\Big\}.
\end{multline}
%
%
%
%
%
%
%
%
%%%%%%%%%%%%%%%%%%%%%%%%%%%%%%%%%%%%%%%%%%%%%%%%%%%%%%%%
%
%\subsection{Properties of the kernel}
First of all we verify that the kernel $\tilde\La$ satisfies the following properties
%\footnote{Since $\La$ is real valued we don't take the complex conjugate in the right-hand side of \eqref{propertiesb}.}
\begin{equation}
\begin{array}{ll}\label{properties}
\tilde\La(k,\ell)=\tilde\La(\ell,k) &\qquad \mbox{(symmetry)},\\
\tilde\La(k,\ell)=\overline{\tilde\La(-k,-\ell)} &\qquad \mbox{(reality)},\\
\tilde\La(\a k,\a\ell)=\a^2\tilde\La(k,\ell) \qquad\forall\,\a>0&\qquad \mbox{(homogeneity)}.
\end{array}
\end{equation}
Considering some particular cases we can considerably simplify $\tilde\La$ as follows
\begin{equation}\label{simpler}
\tilde\La(k,\ell)=\begin{cases}
\ds - (1+\sla)k\ell   &\qquad \mbox{if } k>0,\,\ell>0\, ,\\
\ds  (1+\sla) \ell(k+\l) &\qquad \mbox{if } k+\ell>0,\, \ell<0 \, ,
\end{cases}
\end{equation}
where the values of $\tilde\La$ in other regions of the $(k,\l)$-plane follow from \eqref{properties}, \eqref{simpler}.
$\tilde\La$ can be written in different way as
\begin{equation*}
\begin{array}{ll}\label{lambda1}
\ds \tilde\La(k,\ell)=-(1+\sla) \frac{2|k+\l|\, |k|\, |\l|}{|k+\l|+|k|+|\l|}\,.
\end{array}
\end{equation*}
After an appropriate rescaling in time, we write \eqref{solvability4}, \eqref{simpler} as
%%%%%%%%%%%%%
\begin{equation}
\begin{array}{ll}\label{solvability5}
\ds    \hvpu_\tau(k,\tau) + i\,\sgn(k) \intR \La(k-\ell,\ell )\,\hvpu(k-\ell, \tau)\, \hvpu(\ell, \tau) \, d\ell 
=0\, , \qquad \forall \, k\not=0\, ,
\end{array}
\end{equation}
with the new kernel $\La$ defined by
%%%%%%%%%%%%%%%%%%%%%%%%%%%
\begin{equation}
\begin{array}{ll}\label{lambda}
\ds \La(k,\ell)= \frac{2|k+\l|\, |k|\, |\l|}{|k+\l|+|k|+|\l|}\,.
\end{array}
\end{equation}
Equation \eqref{solvability5}, \eqref{lambda} is well-known as it coincides with the amplitude equation for nonlinear Rayleigh waves \cite{hamilton-et-al} and describes the propagation of surface waves on a tangential discontinuity (current-vortex sheet) in incompressible MHD \cite{ali-hunter}.
The spacial form of \eqref{solvability5}, \eqref{lambda} is, see \cite{ali-hunter-parker,hamilton-et-al}
\begin{equation*}
\begin{array}{ll}\label{}
\vphi^{(1)}_\tau+\frac12 \HH[\Phi^2]_{\theta\theta} + \Phi\vphi^{(1)}_{\theta\theta}=0, \qquad \Phi=\HH[\vphi^{(1)}] \, ,
\end{array}
\end{equation*}
where $\HH$ denotes the Hilbert transform. After renaming of variables it becomes \eqref{equ1}, \eqref{equ1bis}.

%%%%%%%%%%%%%%%%%%%%%%%%%%%%%%%%%%%%%%%%%%%%%%
%%%%%%%%%%%%%%%%%%%%%%%%%%%%%%%%%%%%%%%%%%%%%%
%%%%%%%%%%%%%%%%%%%%%%%%%%%%%%%%%%%%%%%%%%%%%%

$\La$ is perhaps the simplest kernel arising for surface waves. It satisfies the properties
\begin{subequations}\label{hamiltonians}
\begin{align}
\La(k,\ell)=\La(\ell,k) &\qquad \mbox{(symmetry)},\label{symmetry}\\
\La(k,\ell)=\overline{\La(-k,-\ell)} &\qquad \mbox{(reality)},\label{reality}\\
\La(\a k,\a\ell)=\a^2\La(k,\ell) \qquad\forall\,\a>0&\qquad \mbox{(homogeneity)},\label{homogeneity}\\
\La(k+\l,-\ell)=\, \overline{\La(k,\ell)} \qquad\, \forall k,\l\in\R&\qquad \mbox{(Hamiltonian)}.\label{hamiltonian}
\end{align}
\end{subequations}
The value 2 of the scaling exponent in \eqref{homogeneity} is consistent with the dimensional analysis in \cite{ali-hunter-parker} for surface waves. It is shown by Al\`i et al. \cite{ali-hunter-parker} that \eqref{hamiltonian} is a sufficient condition for \eqref{solvability4}, in addition to \eqref{symmetry}, \eqref{reality}, to admit a {\it Hamiltonian structure}, see also \cite{hamilton-et-al,hunter}. 
Other results on equations of the form \eqref{solvability5} are in the papers \cite{benzoni-JFC, benzoni-JFC-tzvetkov, benzoni-rosini, hunter11, hunter-thoo,marcou}.

%%%%%%%%%%%%%%%%%%%%%%%%%%%%%%%%%%%%%%%%%%%%%%
%%%%%%%%%%%%%%%%%%%%%%%%%%%%%%%%%%%%%%%%%%%%%%
%%%%%%%%%%%%%%%%%%%%%%%%%%%%%%%%%%%%%%%%%%%%%%

The results of Sections \ref{asymptotic} to \ref{second order} are summarized in the following theorem.
\begin{theorem}\label{summa}
Assume that $v^0_1,B^0_1,H^0_1$ are as in (3) or (4) of Lemma \ref{lemma1}, and let $\la$ be a real root of \eqref{condlamda2}. Then the solution $U=(\v,\B,q)^T,\ V=(\H,E)^T$, $\vphi$ of \eqref{mhd2}, \eqref{eq:Maxwell2}, \eqref{bc} admits the asymptotic expansion \eqref{expansion}, where the first order terms of the expansion are defined in \eqref{sol12}, \eqref{sol22}, and the second order terms are found from
\eqref{formU2}, \eqref{equa}, \eqref{equb}, \eqref{solS}, \eqref{L1hp} and \eqref{sol221}--\eqref{detaP}. The location of the plasma-vacuum interface is given by 
\[
x_2=\eps\vphi^{(1)}(x_1-\la t,\eps t)+O(\eps^2),
\]
 as $\eps\to0$, with $t=O(\eps^{-1})$ and $\la$ the linearized phase velocity of the surface wave. The Fourier transform of the leading order perturbation $\varphi^{(1)}(\theta,\tau)$ satisfies the amplitude equation \eqref{solvability5}, \eqref{lambda}.

\end{theorem}
We wish to stress that for the existence of surface waves propagating on the plasma-vacuum interface, it is necessary to have a real root $\la$ of \eqref{condlamda2} satisfying \eqref{condlambda}. This is obtained if the basic state $v^0_1,B^0_1,H^0_1$ is as in (3) or (4) of Lemma \ref{lemma1}.

\section{Noncanonical variables and well-posedness}\label{noncanonical}

As in \cite{hunter2} we introduce the noncanonical dependent variable $\psi(\theta,\tau)$ defined by
\begin{equation*}
\begin{array}{ll}\label{}
\psi(\theta,\tau)=|\p_\th|^{1/2}\varphi^{(1)}(\theta,\tau), \qquad  \hat\psi(k,\tau)=|k|^{1/2}\hvpu(k,\tau).
\end{array}
\end{equation*}
Then rewriting equation \eqref{solvability5} in terms of $\psi$ gives
\begin{equation}
\begin{array}{ll}\label{solvability6}
\ds    \hat\psi_\tau(k,\tau) + i\, k \intR S(k-\ell,\ell )\, \hat\psi(k-\ell, \tau)\,  \hat\psi(\ell, \tau) \, d\ell 
=0\, , \qquad \forall \, k\not=0\, ,
\end{array}
\end{equation}
with kernel $S$ given by 
\begin{equation}
\begin{array}{ll}\label{noncanonical2}
\ds S(k,\ell)=\frac{\La(k,\ell)}{|k\l(k+\l)|^{1/2}}.
\end{array}
\end{equation}
We extend the definition of $S$ by setting
\begin{equation}
\label{defS0}
S(k,\ell)=0 \qquad \mbox{ if }\; k\l=0\,.
\end{equation}
$S$ obviously satisfies
\begin{subequations}\label{properties3}
\begin{align}
S(k,\ell)=S(\ell,k) &\qquad \mbox{(symmetry)},\label{properties2a}\\
S(k,\ell)=\overline{S(-k,-\ell)} &\qquad \mbox{(reality)},\label{properties2b}\\
S(\a k,\a\ell)=\a^{1/2}S(k,\ell) \qquad\forall\,\a>0&\qquad \mbox{(homogeneity)},  \label{properties2c}\\
S(k+\l,-\ell)=\, \overline{S(k,\ell)} \qquad\, \forall k,\l\in\R&\qquad \mbox{(Hamiltonian)}.\label{hamiltonian2}
\end{align}
\end{subequations}
The corresponding spatial form of \eqref{solvability6} is 
\begin{equation}
\begin{array}{ll}\label{spatial}
\p_\tau\psi+\p_\th a(\psi,\psi)=0\, ,
\end{array}
\end{equation}
where the bilinear form $a$ is defined by
\begin{equation}
\begin{array}{ll}\label{bilinear}
\ds  \widehat{a(\psi,\phi)}(k,\tau)=\intR S(k-\ell,\ell )\, \hat\psi(k-\ell, \tau)\,  \hat\phi(\ell, \tau) \, d\ell .
\end{array}
\end{equation}
\eqref{spatial} has the form of a nonlocal Burgers equation, like (2.8) in \cite{hunter2}, or (1.1) in \cite{benzoni2009}.

We consider the initial value problem for the noncanonical equation \eqref{spatial}, \eqref{bilinear}, supplemented by an initial condition
\begin{equation}
\begin{array}{ll}\label{ic}
\psi(\theta, 0)= \psi_0(\theta).
\end{array}
\end{equation}
The well-posedness of \eqref{spatial}--\eqref{ic} easily follows by adapting the proof of Hunter \cite{hunter2} (given for the periodic setting) to our case.

\begin{theorem}\label{wellp}
For any $\psi_0\in H^s(\R)$, $s>2$, the initial value problem \eqref{spatial}--\eqref{ic} has a unique local solution 
\[
\psi\in C(I;H^s(\R))\cap C^1(I;H^{s-1}(\R))
\]
defined on the time interval $I=(-\tau_\ast,\tau_\ast)$, where
\begin{equation}
\begin{array}{ll}\label{tmax}
\ds \tau_\ast=\frac1{K_s\|\psi_0\|^{1-2/s}_{L^2(\R)}\|\psi_0\|^{2/s}_{H^s(\R)}},
\end{array}
\end{equation} 
for a suitable constant $K_s$.
\end{theorem}
The well-posedness result of Theorem \ref{wellp} may be easily recast as a similar result for \eqref{solvability5}, \eqref{lambda}.

%%%%%%%%%%%%%%%%%%%%%%%%%%%%%%%%%%%%%%%%
For the proof we need to introduce the homogeneous space $\dot{H}^s(\R)$,
\[
\dot{H}^s(\R)=\left\{ \psi:\R \to \R\, : \, \intR |k|^{2s} |\hat\psi(k)|^2\, dk <+\infty \right\}\, .
\]
As inner product and norm in $\dot H^s$, we use\footnote{If $\phi$ is real then $\overline{\hat\phi(k)}=\hat\phi(-k)$.}
\[
\langle \psi,\phi \rangle_s= \intR |k|^{2s}\hat\psi(k) \hat\phi(-k)\, dk, \qquad \|\psi\|_s= \left( \intR  |k|^{2s} |\hat\psi(k)|^2dk \right)^{1/2}.
\]
In particular we have
\[
 \|\psi\|_{L^2(\R)}=  \|\psi\|_0= \left( \intR   |\hat\psi(k)|^2dk \right)^{1/2}.
\]
As a norm of $H^s(\R)$ we take
\[
 \|\psi\|_{H^s(\R)}= \left( \intR  \left(1+|k|^{2s}\right) |\hat\psi(k)|^2dk \right)^{1/2}.
\]
%%%%%%%%%%%%%%%%%%%%%%%%%%%%%%%%%%%%%%%%
%\vfill
%\eject

\begin{proof}[Proof of Theorem~\ref{wellp}]
We prove an a priori estimate for the solution. The first part of the proof is as in \cite{hunter2}, but we repeat it for the convenience of the reader. Assuming that we have a sufficiently smooth solution $\psi$, from \eqref{spatial}, \eqref{bilinear} we compute for $s\ge0$
\begin{equation}
\begin{array}{ll}\label{prima}
\ds \frac{d\,}{d\tau}\intR |k|^{2s}\hat\psi(k) \hat\psi(-k)\, dk + 2i \iint_{\R^2} k  |k|^{2s}S(k-\ell,\ell )\, \hat\psi(k-\ell, \tau)\,  \hat\psi(\ell, \tau) \,\hat\psi(-k, \tau)\,  d\ell \, dk=0\, .
\end{array}
\end{equation}
By change of variables and the cyclic symmetry of $S(k,\l)$ we prove
\begin{equation}
\begin{array}{ll}\label{cyclic}
\ds 2i \iint_{\R^2} k  |k|^{2s}S(k-\ell,\ell )\, \hat\psi(k-\ell, \tau)\,  \hat\psi(\ell, \tau) \,\hat\psi(-k, \tau)\,  d\ell \, dk
\\=
\ds -2i \iint_{\R^2} (k-\l)  |k-\l|^{2s}S(k-\ell,\ell )\, \hat\psi(k-\ell, \tau)\,  \hat\psi(\ell, \tau) \,\hat\psi(-k, \tau)\,  d\ell \, dk
\\=
\ds -2i \iint_{\R^2} \l  |\l|^{2s}S(k-\ell,\ell )\, \hat\psi(k-\ell, \tau)\,  \hat\psi(\ell, \tau) \,\hat\psi(-k, \tau)\,  d\ell \, dk
\\=
\ds \frac{2i}3 \iint_{\R^2} \left(k |k|^{2s}-(k-\l)  |k-\l|^{2s}- \l  |\l|^{2s} \right) S(k-\ell,\ell )\, \hat\psi(k-\ell, \tau)\,  \hat\psi(\ell, \tau) \,\hat\psi(-k, \tau)\,  d\ell \, dk
\,.
\end{array}
\end{equation}
%%%%%%%%%%%%%%%%%%%%%%%%%%%%%%%%%%%%%%%%
From the definition \eqref{lambda}, \eqref{noncanonical2} it follows
\begin{equation}
\begin{array}{ll}\label{min}
|S(k-\l,\l)|\le \min\{ |k|^{1/2}, |k-\l|^{1/2}, |\l|^{1/2}\}.
\end{array}
\end{equation}
%%%
Assuming $s>0$,  we may apply the estimate (5.2) in \cite{hunter2}
\begin{equation}\label{estmin}
\ds { \left| k |k|^{2s}-(k-\l)  |k-\l|^{2s}- \l  |\l|^{2s} \right| }\\
\leq C_s\left(  |k|^{s} |k-\l|^{s} |\l|+ |k|^{s} |k-\l| |\l|^{s}+ |k| |k-\l|^{s} |\l|^{s} \right).
\end{equation}
From \eqref{prima}--\eqref{estmin}, applying the appropriate bound on each term, we get
\begin{equation}
\begin{array}{ll}\label{}
\ds \left| \frac{d\,}{d\tau}\intR |k|^{2s}\hat\psi(k) \hat\phi(-k)\, dk \right| \le 2C_s
\iint_{\R^2}  |k|^{s} |k-\l|^{s} |\l|^{3/2}\, | \hat\psi(k-\ell, \tau)\,  \hat\psi(\ell, \tau) \,\hat\psi(-k, \tau)|\,  d\ell \, dk\, .
\end{array}
\end{equation}
Using the Cauchy-Schwarz and Young's inequalities gives
\begin{equation}
\begin{array}{ll}\label{98}
\ds \left| \frac{d\,}{d\tau} \|\psi\|_s^2 \right| \le 2C_s
\| |k|^s\hat\psi\|_{L^2(\R)} \, \| (|k|^s\hat\psi) \ast (|\l|^{3/2}\hat\psi) \|_{L^2(\R)}
 \le 2C_s
\|\psi\|_s^2 \, \| |\l|^{3/2}\hat\psi \|_{L^1(\R)}.
\end{array}
\end{equation}
Applying estimate \eqref{interpol} for $p=s-3/2>1/2, q=-3/2$, yields
\begin{equation}
\begin{array}{ll}\label{}
\ds \| |\l|^{3/2}\hat\psi \|_{L^1(\R)} \leq C \|\psi\|_0^{1-2/s} \|\psi\|_{s}^{2/s},
\end{array}
\end{equation}
and substituting in \eqref{98} gives
\begin{equation}
\begin{array}{ll}\label{100}
\ds \left| \frac{d\,}{d\tau} \|\psi\|_s^2 \right| 
 \le 2CC_s
 \|\psi\|_0^{1-2/s} \|\psi\|_{s}^{2+2/s}.
\end{array}
\end{equation}
%%%
If $s=0$,  the last equality in \eqref{cyclic} shows that such double integral equals zero. It follows from \eqref{prima} that
\begin{equation}
\begin{array}{ll}\label{seconda}
\ds \frac{d\,}{d\tau}\intR  \hat\psi(k) \hat\psi(-k)\, dk = \frac{d\,}{d\tau} \|\psi \|_0^2 =0\, ,
\end{array}
\end{equation}
which gives $\|\psi(t)\|_0=\|\psi_0\|_0$. Combining with \eqref{100} and simplifying the equation gives
\begin{equation}
\begin{array}{ll}\label{terza}
\ds \left| \frac{d\,}{d\tau} \|\psi\|_{H^s(\R)} \right| 
 \le CC_s
 \|\psi_0\|_0^{1-2/s} \|\psi\|_{H^s(\R)}^{1+2/s}.
\end{array}
\end{equation}
Using Gronwall's inequality, we deduce from \eqref{terza} the bound
\begin{equation}
\begin{array}{ll}\label{stima}
\ds   \|\psi(\cdot,\tau)\|_{H^s(\R)} \le \|\psi_0\|_{H^s(\R)}\left( 1- \frac{2CC_s}{s} \|\psi_0\|^{1-2/s}_{L^2(\R)}\|\psi_0\|^{2/s}_{H^s(\R)}|\tau| \right)^{-s/2},
\end{array}
\end{equation}
for $|\tau|<\tau_\ast$ where $\tau_\ast$ is given by \eqref{tmax}.
Given the a priori estimate \eqref{stima}, the proof proceeds by standard arguments, see \cite{kato, taylor97}.
\end{proof}

%%%%%%%%%%%%%%%%%%%%%%%%%%%%%%%%%%%%%%%%%%%%%%%%%%
%%%%%%%%%%%%%%%%%%%%%%%%%%%%%%%%%%%%%%%%%%%%%%%%%%
%%%%%%%%%%%%%%%%%%%%%%%%%%%%%%%%%%%%%%%%%%%%%%%%%%
%%%%%%%%%%%%%%%%%%%%%%%%%%%%%%%%%%%%%%%%%%%%%%%%%%

From the proof of Theorem  \ref{wellp} we can also obtain the following {\it blow-up criterion}.
\begin{lemma}\label{blow}
Under the assumptions of Theorem \ref{wellp}, if $\psi\in C(0,T;H^s(\R))$ with $0<T<+\infty$ is a solution of \eqref{spatial} such that
\begin{equation}
\begin{array}{ll}\label{bound}
\ds \int_0^T\|\psi(\cdot,\tau) \|^{2/s'}_{s'} d\tau<+\infty
\end{array}
\end{equation}
for some $s'>2$, then $\psi$ is continuable to a solution $\psi\in C(0,T';H^s(\R))$ with $T'>T$.
\end{lemma}
\begin{proof}
Applying estimate \eqref{interpol} for $p=s'-3/2>1/2, q=-3/2$, yields
\begin{equation*}
\begin{array}{ll}\label{}
\ds \| |\l|^{3/2}\hat\psi \|_{L^1(\R)} \leq C \|\psi\|_0^{1-2/s'} \|\psi\|_{s'}^{2/s'},
\end{array}
\end{equation*}
and substituting in \eqref{98} gives
\begin{equation*}
\begin{array}{ll}\label{}
\ds \left| \frac{d\,}{d\tau} \|\psi\|_s^2 \right| 
 \le 2CC_s
 \|\psi\|_0^{1-2/s'} \|\psi\|_{s'}^{2/s'} \|\psi\|_{s}^{2}.
\end{array}
\end{equation*}
%%%
Combining with \eqref{seconda} we readily obtain
\begin{equation}
\begin{array}{ll}\label{quinta}
\ds \left| \frac{d\,}{d\tau} \|\psi\|_{H^s(\R)}^2 \right| 
 \le CC_s
 \|\psi_0\|_0^{1-2/s'} \|\psi\|_{s'}^{2/s'}\|\psi\|^2_{H^s(\R)}.
\end{array}
\end{equation}
Applying the Gronwall inequality with \eqref{bound} gives the thesis.
\end{proof}

The thesis of Lemma \ref{blow} can also be obtained by directly assuming $|\l|^{3/2}\hat\psi \in L^1((0,T)\times\R)$, instead of \eqref{bound}, as immediately follows from the Gronwall inequality applied to \eqref{98}.
This second blow up criterion is the analogue of that one in \cite{benzoni2009} for a homogeneous kernel of order $1/2$.

%\vfill
%\eject

%%%%%%%%%%%%%%%%%%%%%%%%%%%%%%%%%%%%%%%%%%%%%%%%%%%%%%%
%%%%%%%%%%%%%%%%%%%%%%%%%%%%%%%%%%%%%%%%%%%%%%%%%%%%%%%
%%%%%%%%%%%%%%%%%%%%%%%%%%%%%%%%%%%%%%%%%%%%%%%%%%%%%%%

\appendix
\section{}\label{appA}
\begin{lemma}\label{}
For all $p,q\in\R$, $q<1/2<p$, there exists a positive constant $C_{p,q}$ such that for all functions $\psi\in \dot{H}^{p+3/2}(\R) \cap \dot{H}^{q+3/2}(\R)$ there holds
\begin{equation}
\begin{array}{ll}\label{interpol}
\ds \| |\l|^{3/2}\hat\psi \|_{L^1(\R)} \leq C_{p,q} \|\psi\|_{q+3/2}^{\frac{p-1/2}{p-q}} \|\psi\|_{p+3/2}^{\frac{1/2-q}{p-q}}
\end{array}
\end{equation}

\end{lemma}
\begin{proof}
For $L>0$, we compute
\begin{multline*}
\ds \| |\l|^{3/2}\hat\psi \|_{L^1(\R)} =  \int_{|\l|\le L} |\l|^{-q} |\l|^{3/2+q}|\hat\psi (\l) |\, d\l + \int_{|\l|\ge L}  |\l|^{-p}|\l|^{3/2+p}|\hat\psi (\l) |\, d\l \\
\le \left(  \int_{|\l|\le L} |\l|^{-2q} \, d\l \right)^{1/2}  \left(  \int_{|\l|\le L} |\l|^{3+2q}|\hat\psi (\l) |^2\, d\l \right)^{1/2}
\\ +
\left(  \int_{|\l|\ge L} |\l|^{-2p} \, d\l \right)^{1/2}  \left(  \int_{|\l|\ge L} |\l|^{3+2p}|\hat\psi (\l) |^2\, d\l \right)^{1/2} 
\\
\le C_qL^{1/2-q}\|\psi\|_{q+3/2} + C_pL^{1/2-p}\|\psi\|_{p+3/2} ,
\end{multline*}
where we have used the assumption $q<1/2<p$. Choosing $L$ such that 
\[C_qL^{1/2-q}\|\psi\|_{q+3/2}= C_pL^{1/2-p}\|\psi\|_{p+3/2}
\]
gives \eqref{interpol}.
\end{proof}

%%%%%%%%%%%%%%%%%%%%%%%%%%%%%%%%%%%%%%%%%%%%%%%%%%%%%%%
%%%%%%%%%%%%%%%%%%%%%%%%%%%%%%%%%%%%%%%%%%%%%%%%%%%%%%%
%%%%%%%%%%%%%%%%%%%%%%%%%%%%%%%%%%%%%%%%%%%%%%%%%%%%%%%

%\bibliographystyle{plain}
%\bibliography{amplie}

\begin{thebibliography}{10}

\bibitem{ali-hunter}
G.~Al{\`{\i}} and J.~K. Hunter.
\newblock Nonlinear surface waves on a tangential discontinuity in
  magnetohydrodynamics.
\newblock {\em Quart. Appl. Math.}, 61(3):451--474, 2003.

\bibitem{ali-hunter-parker}
G.~Al{\`{\i}}, J.~K. Hunter, and D.~F. Parker.
\newblock Hamiltonian equations for scale-invariant waves.
\newblock {\em Stud. Appl. Math.}, 108(3):305--321, 2002.

\bibitem{benzoni2009}
S.~Benzoni-Gavage.
\newblock Local well-posedness of nonlocal {B}urgers equations.
\newblock {\em Differential Integral Equations}, 22(3-4):303--320, 2009.

\bibitem{benzoni-JFC}
S.~Benzoni-Gavage and J.-F. Coulombel.
\newblock On the amplitude equations for weakly nonlinear surface waves.
\newblock {\em Arch. Ration. Mech. Anal.}, 205(3):871--925, 2012.

\bibitem{benzoni-JFC-tzvetkov}
S.~Benzoni-Gavage, J.-F. Coulombel, and N.~Tzvetkov.
\newblock Ill-posedness of nonlocal {B}urgers equations.
\newblock {\em Adv. Math.}, 227(6):2220--2240, 2011.

\bibitem{benzoni-rosini}
S.~Benzoni-Gavage and M.~D. Rosini.
\newblock Weakly nonlinear surface waves and subsonic phase boundaries.
\newblock {\em Comput. Math. Appl.}, 57(9):1463--1484, 2009.

\bibitem{BFKK}
I.~B. Bernstein, E.~A. Frieman, M.~D. Kruskal, and R.~M. Kulsrud.
\newblock An energy principle for hydromagnetic stability problems.
\newblock {\em Proc. Roy. Soc. London. Ser. A.}, 244:17--40, 1958.

\bibitem{CDAS}
D.~Catania, M. D'Abbicco \& P.~Secchi.
\newblock Well-posedness of the linearized mhd-maxwell free boundary problem.
\newblock {\em Preprint 2013}.

\bibitem{Goed}
J.P. Goedbloed, S.~Poedts.
\newblock {\em Principles of magnetohydrodynamics with applications to
  laboratory and astrophysical plasmas}.
\newblock Cambridge University Press, Cambridge, 2004.

\bibitem{hamilton-et-al}
M.F. Hamilton, Yu.~A. Il'insky, and E.~A. Zabolotskaya.
\newblock Evolution equations for nonlinear {R}ayleigh waves.
\newblock {\em J. Acoust. Soc. Am.}, 97:891--897, 1995.

\bibitem{hunter}
J.~K. Hunter.
\newblock Nonlinear surface waves.
\newblock In {\em Current progress in hyberbolic systems: {R}iemann problems
  and computations ({B}runswick, {ME}, 1988)}, pages 185--202. Amer. Math.
  Soc., 1989.

\bibitem{hunter2}
J.~K. Hunter.
\newblock Short-time existence for scale-invariant {H}amiltonian waves.
\newblock {\em J. Hyperbolic Differ. Equ.}, 3(2):247--267, 2006.

\bibitem{hunter11}
J.~K. Hunter.
\newblock Nonlinear hyperbolic surface waves.
\newblock In {\em Nonlinear conservation laws and applications}, volume 153 of
  {\em IMA Vol. Math. Appl.}, pages 303--314. Springer, New York, 2011.

\bibitem{hunter-thoo}
J.~K. Hunter and J.~B. Thoo.
\newblock On the weakly nonlinear {K}elvin-{H}elmholtz instability of
  tangential discontinuities in {MHD}.
\newblock {\em J. Hyperbolic Differ. Equ.}, 8(4):691--726, 2011.

\bibitem{kato}
T.~Kato.
\newblock Nonlinear equations of evolution in {B}anach spaces.
\newblock In {\em Nonlinear functional analysis and its applications, {P}art
  2}, volume~45 of {\em Proc. Sympos. Pure Math.}, pages 9--23. Amer. Math.
  Soc., 1986.

\bibitem{kreiss}
H.~O. Kreiss.
\newblock Initial boundary value problems for hyperbolic systems.
\newblock {\em Comm. Pure Appl. Math.}, 23:277--298, 1970.

\bibitem{mandrik-trakhinin}
N.~Mandrik, Y.~Trakhinin.
\newblock Influence of vacuum electric field on the stability of a
  plasma-vacuum interface.
\newblock {\em Commun. Math. Sci.}, 12(6):1065--1100, 2014.

\bibitem{marcou}
A.~Marcou.
\newblock Rigorous weakly nonlinear geometric optics for surface waves.
\newblock {\em Asymptot. Anal.}, 69(3-4):125--174, 2010.

\bibitem{mttpv}
A.~Morando, Y.~Trakhinin, and P.~Trebeschi.
\newblock Well-posedness of the linearized plasma-vacuum interface problem in
  ideal incompressible mhd.
\newblock {\em Quart. Appl. Math.}, to appear, 2013.

\bibitem{SeTr}
P.~Secchi, Y.~Trakhinin.
\newblock Well-posedness of the linearized plasma-vacuum interface problem.
\newblock {\em Interfaces Free Bound.}, 15(3):323--357, 2013.

\bibitem{SeTrNl}
P.~Secchi, Y.~Trakhinin.
\newblock Well-posedness of the plasma-vacuum interface problem.
\newblock {\em Nonlinearity}, 27:105--169, 2014.

\bibitem{taylor97}
Michael~E. Taylor.
\newblock {\em Partial differential equations. {III}}, volume 117 of {\em
  Applied Mathematical Sciences}.
\newblock Springer-Verlag, New York, 1997.
\newblock Nonlinear equations, Corrected reprint of the 1996 original.

\bibitem{trakhinin10}
Y.~Trakhinin.
\newblock On the well-posedness of a linearized plasma-vacuum interface problem
  in ideal compressible mhd.
\newblock {\em J. Differential Equations}, 249:2577--2599, 2010.

\bibitem{trakhinin12}
Y.~Trakhinin.
\newblock Stability of relativistic plasma-vacuum interfaces.
\newblock {\em J. Hyperbolic Differential Equations}, 9:469--509, 2012.

\end{thebibliography}

\end{document}